\documentclass{article} 

\usepackage[utf8]{inputenc} 
\usepackage[english]{babel}

\usepackage{geometry} 
\usepackage{graphicx} 
\usepackage{amsmath}
\usepackage{amssymb}

\usepackage{amsthm}

\usepackage{algorithm}
\usepackage{algpseudocode}
\usepackage{hyperref}

\usepackage{booktabs} 
\usepackage{array} 
\usepackage{paralist} 
\usepackage{verbatim} 
\usepackage{subfig} 
\usepackage{amsfonts}
\usepackage{fancyhdr} 
\usepackage{sectsty}
\allsectionsfont{\sffamily\mdseries\upshape} 

\usepackage[nottoc,notlof,notlot]{tocbibind} 
\usepackage[titles,subfigure]{tocloft} 

\usepackage{mathrsfs}
\usepackage[all,cmtip]{xy}

\newcommand{\ZZ}{\mathbb{Z}}
\newcommand{\RR}{\mathbb{R}}
\newcommand{\QQ}{\mathbb{Q}}

\newcommand{\PP}{\mathbb{P}}

\newcommand{\OO}{\mathcal{O}}

\newtheorem{question}{Question}
\newtheorem*{question*}{Question}
\newcommand{\spec}{\text{Spec}\hspace{0.5mm}
}

\newcommand{\McE}{\mathcal{E}}

\newcommand{\McP}{\mathcal{P}}

\newcommand{\ra}{\rightarrow}

\newcommand{\Alb}{\textnormal{Alb}}

\newcommand{\McM}{\mathcal{M}}

\newcommand{\Pic}{\textnormal{Pic}}
\newcommand{\Nef}{\textnormal{Nef}}

\newcommand{\NE}{\overline{\textnormal{NE}}}

\newcommand{\Aut}{\textnormal{Aut}}
\newcommand{\Num}{\equiv_\textnormal{num}}
\newcommand{\Lin}{\equiv_\textnormal{lin}}

\newcommand{\pp}{\prime}

\newcommand{\Qb}{\overline{\mathbb{Q}}}
\newcommand{\Send}{\textnormal{SEnd}}
\newcommand{\Iamp}{\textnormal{IAmp}}

\newtheorem*{tracmmp}{The tractable minimal model program for a dynamical property $\mathcal{D}$}

\newtheorem{theorem}{Theorem}[section]
\newtheorem*{theorem*}{Theorem}
\newtheorem*{conjecture*}{Conjecture}
\newtheorem{conjecture}{Conjecture}
\newtheorem*{proposition*}{Proposition}

\newtheorem{corollary}{Corollary}[theorem]
\newtheorem{lemma}[theorem]{Lemma}
\newtheorem*{lemma*}{Lemma}
\newtheorem{definition}[theorem]{Definition}
\newtheorem{proposition}[theorem]{Proposition}

\newtheorem{example}{Example}
\newtheorem*{example*}{Example}

\newtheorem*{remark*}{Remark}

\newtheorem*{Ingredients*}{Ingredients of the minimal model program for $K_X$}

\newtheorem{Case Q}{Case Q}

\usepackage{color}
\newif\ifhascomments \hascommentstrue
\ifhascomments
  \newcommand{\matt}[1]{{\color{red}[[\ensuremath{\spadesuit\spadesuit\spadesuit} #1]]}}
  \newcommand{\brett}[1]{{\color{blue}[[\ensuremath{\heartsuit\heartsuit\heartsuit} #1]]}}
\else
  \newcommand{\matt}[1]{}
  \newcommand{\brett}[1]{}
\fi

\author{Brett Nasserden\\ Department of Mathematics, Western University\\ \href{mailto:bnasserd@uwo.ca}{bnasserd@uwo.ca} }
\title{Some applications of the minimal model program in arithmetic dynamics}
\date{}

\begin{document}
\maketitle
\begin{abstract}
 We describe a general program for studying the dynamics of surjective endomorphisms of algebraic varieties that are amenable to techniques from the minimal model program. We obtain density results on the pre-periodic points of surjective endomorphisms of varieties admitting an int-amplified endomorphism, and reduce certain cases of the Medvedev-Scanlon conjecture to so called Q-abelian varieties using our approach. We also provide a connection between the existence of an automorphism with positive entropy and group of connected components of a variety. In particular, we show that if $X$ is normal and projective with finitely generated nef cone then $X$ has an automorphism of positive entropy if and only if the group of connected components $\pi_0\Aut(X)$ has an element of infinite order.
\end{abstract}

\begin{section}{Introduction}
One of the most potent tools available to study the geometry of higher dimensional algebraic varieties is the minimal model program initiated by Mori. However, when one tries to use the minimal model program to study to study the dynamics of surjective \emph{morphisms} of varieties the following problem arises. If $f\colon X\ra X$ is a surjective endomorphism and $\pi \colon X\ra Y$ is a contraction morphism, then apriori $f$ descends to a dominant rational mapping $f\colon Y\dashrightarrow Y$ which need not extend to an everywhere defined morphism on $Y$. One must either be content to work with dominant rational maps, or seek situations that guarantee that $f$ extends. As morphisms are significantly better behaved than dominant rational maps, we will pursue this latter case. 

Our goal in this article is to describe a method of obtaining at partial results in arithmetic dynamics for classes of varieties where the operations of the minimal model program can be applied to study endomorphisms. The main examples of such varieties are varieties that admit an \emph{int-amplified} endomorphism. For example, all toric varieties and abelian varieties admit an int-amplified endomorphism. The existence of a single int-amplified endomorphism on a projective variety $X$ drastically effects its birational geometry, so much to the extent that it makes it possible to use the minimal model program to study all surjective endomorphisms of $X$, even the non int-amplified endomorphisms. This approach was pioneered by Zhang in \cite{MR3742591} and extended in \cite{MR4117085} with Meng.

We first describe the general approach and then describe three applications of these ideas. In \ref{sec:preperiodicpoints} we study when a morphism has a dense set of a pre-periodic points. In \ref{sec:MSC} we consider the Medvedev-Scanlon conjecture from the perspective of the	minimal model program. Finally, in \ref{sec:automorphisms} consider the Kawaguchi-Silverman conjecture for automorphisms and prove a criterion for when a variety with a finitely generated nef cone to have an automorphism of positive entropy. 

\begin{subsection}{Our results.}
In  section \ref{sec:preperiodicpoints} we apply these ideas to study when a surjective morphism has a dense set of pre-periodic points. When a morphism amplified, namely when $f^*L=qL$ for some $q>0$ and $L$ ample then Fakhruddin showed in \cite{MR1995861} that there is a dense set of pre-periodic points. However, in some situations it is too restrictive to assume that a moprhism is amplified. 
\begin{theorem}
Let $X$ be a $\QQ$-factorial variety with at worst terminal singularities that admits an int-amplified endomorphism. Suppose further that $X$ is rationally connected with $\Alb(X)=0$. Let $f\colon X\ra X$ be a surjective endomorphism. Then $f$ has a dense set of pre-periodic points.    
\end{theorem}
We also illustrate our ideas with the example of toric morphisms and the toric minimal model program to obtain
\begin{theorem}
		Let $X$ be a $\QQ$-factorial projective toric variety defined over $\overline{\QQ}$. Let $f\colon X\ra X$ be a surjective toric morphism. Then the set of pre-periodic points of $f$ is Zariski dense. 
	\end{theorem}
One could likely obtain this result by direct methods, its utility is as a proof of concept.	

We now turn to the Medvedev-Scanlon conjecture. \begin{definition}
			Let $X$ be a projective variety and suppose that $f\colon X\ra X$ is a surjective endomorphism. We say that $f$ is \emph{fiber preserving} if there is a positive dimensional variety $Z$ and a dominant rational map $\psi\colon X\dashrightarrow Z$ such that $\psi\circ f=\psi$. 
		\end{definition}
		The conjecture is usually stated as follows.
		\begin{conjecture}[The Medvedev-Scanlon conjecture: See conjecture 7.14 \cite{MR3126567}]
			Let $X$ be an irreducible variety defined over an algebraically closed field $F$ of characteristic 0. Let $\phi\colon X\dashrightarrow X$ be a dominant rational map. If $\phi$ is not fiber preserving then there is a point $x\in X(F)$ with a forward dense orbit under $\phi$.
		\end{conjecture}
We describe an approach using the minimal model program to tackle the certain cases of the Medvedev-Scanlon conjecture. While our results are preliminary we obtain the following. Recall that an $\emph{Q-Abelian}$ variety is a normal projective variety $X$ that admits a finite surjection $h\colon A\ra X$.	

\begin{theorem}
Assume that the Medvedev-Scanlon conjecture holds for surjective endomorphisms of all $Q$-abelian varieties. Then the Medvedev-Scanlon conjecture holds for all $\QQ$-factorial normal projective varieties with at worst terminal singularities that satisfy the following two conditions:
\begin{enumerate}
	\item $X$ admits an int-amplified endomorphism.
	\item $\kappa(X)= 0$. 
\end{enumerate}
\end{theorem}
In other words, the  Medvedev-Scanlon conjecture for Q-abelian varieties would give interesting results on the Medvedev-Scanlon conjecture for all varieties admitting an int-amplified endomorphism of Kodaira dimension 0. 

Finally in section 
\ref{sec:automorphisms} we consider automorphisms of projective varieties from this perspective. This leads to the following result.

\begin{theorem}
			Let $X$ be a normal projective variety over $\Qb$ with a finitely generated nef cone. Let $f\in \pi_0\Aut(X)$. Then $\lambda_1(f)>1\iff$ f has infinite order in $\pi_0\Aut(X)$. In particular a normal projective variety $X$ with finitely generated nef cone has an automorphism of positive entropy if and only if $\pi_0 \Aut(X)$ has an element of infinite order.
		\end{theorem}	
Here $\lambda_1(f)$ is the \emph{dynamical degree} of $f$ which is the spectral raduis of the pull back morphism $f^*\colon N^1(X)\ra N^1(X)$.

While many of the results in these sections are preliminary, we hope that they will bear fruit in the future.

\end{subsection}

\end{section}

\begin{section}{Dynamical Preliminaries}

Given a projective variety $X$ defined over $\Qb$ (more generally any algebraically closed field) it is an old question to study the group of symmetries of $X$, in other words its automorphisms. One might also generalize this notion and instead of studying $\Aut(X)$ one might study $\textnormal{Sur}(X)$. That is, the monoid of all surjective self morphisms $X\ra X$. It turns out that when $\textnormal{Sur}(X)$ is strictly larger then $\Aut(X)$ we should expect that $X$ has a special geometry. Evidence for this meta-principle is the following.

\begin{theorem}[Nakayama Classification of surfaces: \cite{MR1934136}]\label{thm: Nakayamaclassification}
Let $X$ be a smooth projective surface defined over $\Qb$. Suppose that $f\colon X\ra X$ is a surjective endomorphism that is not an autmorphism. 

\begin{enumerate}
    \item If $\kappa(X)\geq 0$ then $X$ is an abelian surface, a hyper-elliptic surface, or a minimal elliptic surface with $\kappa(X)=1$ and $\chi(\OO_X)=0$.
    \item If $X$ is a ruled surface that $X$ is one of the following.
    \begin{enumerate}
        \item A toric surface.
        \item A $\PP^1$-bundle over an elliptic curve.
        \item A $\PP^1$-bundle over a smooth genus $g>1$ curve $C$ that is trivialized after a finite etale base change. 
    \end{enumerate}
\end{enumerate}
	
\end{theorem}

Unfortunately results this strong are not known in higher dimensions. However there is some work in \cite{MR2359102}. The main two examples of varieties that admit surjective endomorphisms that are not automorphisms are Abelian varieties and toric varieties.

\begin{section}{Surjective Endomorphisms of projective varieties}\label{sec:surjectivemorphisms}
Here we develop the theory of surjective endomorphisms of projective varieties. Most results are easy and folklore but we include for lack of a suitable reference for certain needed results. 
	
	\begin{proposition}[Dominant morphism of projective varieties is surjective]\label{prop:dom=sur}
		Let $X$ be irreducible projective varieties defined over $\Qb$. Let $f\colon X\ra Y$ be a dominant morphism. Then $f$ is surjective.
	\end{proposition}
	\begin{proof}
		The morphism $f$ factors as $X\mapsto \Gamma_f\mapsto Y$ where $\Gamma_f\subseteq X\times Y$ is the graph of $f$. Since the morphism $X\ra \Gamma_f$ is a closed immersion and $X\times Y\ra Y$ is closed as $X,Y$ are projective we have that the image of $f$ is closed. Since $f$ is dominant it must be the whole of $Y$. Alternatively, one may use that $f$ must be proper, and therefore is closed.  
	\end{proof}

	\begin{proposition}\label{prop:surjective=finite}
		Let $X$ be an irreducible projective varieties defined over $\Qb$. Let $f\colon X\ra X$ be a dominant morphism, then $f$ is finite.
	\end{proposition}
	\begin{proof}
		Consider $N_1(X)_\QQ$. Since $f_*\colon N_1(X)_\QQ\ra N_1(X)_\QQ$ is surjective, it is also injective as  $N_1(X)_\QQ$ is a finite dimensional vector space. It follows that $f_*$ does not contract any curve. In particular, that $f$ has finite fibers. Since $f$ is a proper morphism with finite fibers it is finite as needed. 
	\end{proof}
	
	\begin{lemma}[exercise 12.22 \cite{MR4225278}]\label{lem:suropen}
		Let $f\colon X\ra Y$ be a dominant integral endomorphism between integral schemes. Assume $Y$ is normal. Then $f$ is universally open. 
	\end{lemma}
	\begin{proof}
		We sketch the proof that $f$ is open. Let $\spec B$ be an open affine in $X$. Since $f$ is finite it is affine and $f^{-1}(\spec B)=\spec A$. If $f\colon \spec A\ra \spec B$ is open for all $\spec B$ then $f$ is open. So we may assume that $f$ is induced by $f\colon \textnormal{Spec }A\ra \textnormal{Spec }B$ which is dominant and finite and $A,B$ are normal domains. So we must prove the statement for a finite extension of rings $B\subseteq A$. Let $a\in A$ be non-zero. Then we can find $b_0,...,b_{n-1}\in B$ such that 
		\[b_0+b_1a+...+b_{n-1}a^{n-1}+a^n=0\]
		with $n$ minimal. Let $p=T^n+\sum_{i=0}^{n-1}b_iT^i$. We will show that $f(D(a))=\bigcup_{i=0}^{n-1} D(b_i)$. To this end let $\mathfrak{p}\in D(a)$. Then $f(\mathfrak{p})=\mathfrak{p}\cap B$. Since $\mathfrak{p}\in D(a)$ we have that $a\notin \mathfrak{p}$. If $\mathfrak{p}\cap B\notin D(b_i)$ for all $i$ then we have that $b_i\in \mathfrak{p}$ for all $i$. Then 
		\[a^n=-b_ia^{n-1}-...-b_0\in \mathfrak{p}.\]
		So $a^n\in \mathfrak{p}$ which means $a\in \mathfrak{p}$ as $\mathfrak{p}$ is a prime ideal. So $f(\mathfrak{p})\in D(b_i)$ for some $i$. On the other hand, let $\mathfrak{q}\in \bigcup_{i=0}^{n-1}D(b_i)$. First suppose that $a\in B$. Then we have that $a$ satisfies $T-a$ and so by the minimality assumption that $p(T)=T-a$. Then $a\notin \mathfrak{p}$ and $f^{-1}(D_B(a))=D(a)$ here $D_B(a)=\{\mathfrak{q}\in \spec B: a\notin \mathfrak{q}\}$.So we have the result. Therefore we may assume that $n>1$ and $a\notin \mathfrak{q}$. Pick $j$ with $b_j\notin \mathfrak{q}$. Since $f$ is surjective we can find a prime ideal $\mathfrak{p}\subseteq A$ with $\mathfrak{p}\cap B=\mathfrak{q}$. If $a\in \mathfrak{p}$. I claim that $a\notin \mathfrak{p}$. Towards a contradiction suppose that $a\in \mathfrak{p}$. Then we have that
		\[a^n+b_{n-1}a^{n-1}+...+b_0=0\] 
		so that $b_0\in \mathfrak{p}\Rightarrow b_0\in \mathfrak{q} $. Then we have that $a(a^{n-1}+b_{n-1}a^{n-2}+...+b1)\in \mathfrak{q}$. Since $a\notin \mathfrak{q}$ we have $a^{n-1}+b_{n-1}a^{n-2}+...+b1\in \mathfrak{q}\Rightarrow b_1\in \mathfrak{q}$. Continuing on in this way we obtain that that $b_0,...,b_{n-1}\in \mathfrak{q}$. This contradicts our choice that some $b_i\notin\mathfrak{q}$. So $\mathfrak{p}$ is in $D(a)$ and so if $\mathfrak{q}\in \bigcup_{i=0}^{n-1}D(b_i)$ then we have that if $\mathfrak{p}\cap B=\mathfrak{q}$ then there is some $\mathfrak{p}\in \spec A$ with $\mathfrak{p}\cap B=\mathfrak{q}$. Thus $ \bigcup_{i=0}^{n-1}D(b_i)\subseteq f(D(a))$. As we already showed that $\bigcup_{i=0}^{n-1}D(b_i)\supseteq f(D(a))$ we have equality.
	\end{proof}
	
	Given a projective variety $X$ and a surjective endomorphism $f\colon X\ra X$ unless $X$ has some very special geometry, (such as being an abelian variety or a toric variety) it can be difficult to study both $f$ and $X$ itself. One way to study $X$ is to apply the minimal model program to $X$ and thus obtain a simpler model of $X$. This strategy also may be employed to study $f$ in certain situations. Let us begin with the case of $X$ being a Mori-fiber space. In other words there is a fibering type contraction $\pi \colon X\ra Y$ where $\dim Y<\dim X$ and $\rho(Y)=\rho(X)-1$. In this situation, $f$ descends to $Y$ after iterating $f$.
	
	\begin{lemma}[{\cite[Lemma 6.2]{1802.07388}}]\label{lemma:iterationlemma}
		Let $\pi\colon X\ra Y$ be a Mori-fiber space. Suppose that $f\colon X\ra X$ is a surjective endomorphism. Then there is some iterate $f^n\colon X\ra X$ and $g\colon Y\ra Y$ such that
		\[\xymatrix{X\ar[r]^{f^n}\ar[d]_\pi & X\ar[d]^\pi\\ Y\ar[r]_g & Y}\]
		commutes. 
	\end{lemma}
	We see that if $X$ is a Mori-fiber space then, then we can study $f$ up to taking iterates by studying the induced morphism $g\colon Y\ra Y$. As $Y$ is "simpler" than $X$ we hope that $g$ is easier to study then $X$. The philosophy is that a good way to study $f$ is to study $g$ and the behavior of $f$ on the fibers of $\pi$. To study birational contractions we have the following two crucial results.
	\begin{lemma}[{\cite[Lemma 3.6]{1902.06071}}]\label{lem:iterationlemma2}
		Let $X$ be a normal $\QQ$-factorial log canonical projective variety and $f\colon X\ra X$ a surjective endomorphism. Let $R$ be a $K_X$ negative extremal ray and $f_*R=R$. Let $\phi_R\colon X\ra Y$ be the associated extremal contraction. Then there is a morphism $g\colon Y\ra Y$ such that $g\circ \phi_R=\phi_R\circ f$.
	\end{lemma}
	
In the above lemma, log canonical is a generalization of a terminal singularity. See \cite{MM} for more on these definitions. This gives us good control over a birational morphism. However, we also need to know how to control flips. This is provided by the following crucial result of Zhang with proof suggested by N. Nakayama
	
	\begin{lemma}[Morphism extension property: lemma 6.6 of \cite{MR4074056}]\label{lem:extensionlem2}
		Let $X$ be a normal projective variety with at worst lc singularities. Suppose that $f\colon X\ra X$ is a surjective endomorphism. Let $R$ be a $K_X$ negative extremal ray and $\phi_R$ the associated contraction. If $\phi_R\colon X\ra Y$ is of flipping type and $\psi\colon X\ra X^+$ is the associated flip, then the induced rational mapping $f^+\colon X^+\dashrightarrow X^+$ extends to a morphism $f^+\colon X^+\ra X^+$. Furthermore, both $f$ and $f^+$ descend to the same morphism on $Y$. 
	\end{lemma}
Thus if $\overline{\textnormal{NE}}(X)_\RR$ is a finitely generated cone and $R$ is an extremal ray generating an extremal contraction $\phi_R$ then we have that the linear mapping preserves $\overline{\textnormal{NE}}(X)$. Therefore for some $n>0$ we have that $f^{\circ n}_*(R)=R$ and so we may obtain a diagram
    \[\xymatrix{X\ar[r]^{f^{\circ n}}\ar[d]_{\phi_R} & X\ar[d]^{\phi_R}\\
		Y\ar[r]_g & Y}\]
    So to study $f$ we wish to study $g$ and proceed by induction. This breaks the dynamical study of $f$ into three essential ingredients.
	
	\begin{enumerate}\label{theplan}
		\item Completing the minimal model program. In particular  the termination of flips. 
		\item The study of $f$ under birational extremal contractions.
		\item The study of the dynamics of Mori-fiber spaces.
	\end{enumerate}
	
	The above program is was is designed to study varieties with finitely generated Nef cones. However it is possible to extend these ideas to a different setting where non-finitely generated Nef cones are allowed. These are varieties with an int-amplified endomorphisms. Note that varieties with an int-amplified endomorphism does not contain all varieties with finitely generated rational Nef cones; we will exhibit examples of varieties with finitely generated rational Nef cones that do not admit an int-amplified endomorphism.

	\begin{subsection}{Int-amplified endomorphisms}\label{subsec:intamplified}
		
		\begin{definition}[Int amplified, amplified and polarized endomorphisms]
			Let $X$ be a projective variety defined over $\Qb$ and $f\colon X\ra X$ a surjective endomorphism. We say that $f$ is a \emph{polarized} endomorphism if there is some ample $\QQ$-Cartier divisor $L$ on $X$ such that $f^*L\Lin qL$ for some $q>1$. We say that $f$ is \emph{amplified} if there is a $\QQ$-Cartier divisor $L$ such that $f^*L-L$ is ample. We say $f$ is an  \emph{int-amplified} endomorphism if there is some ample $\QQ$-Cartier divisor $L$ with $f^*L-L$ being ample.
		\end{definition}
		
		Int amplified endomorphisms can be characterized in terms of their eigenvalues being large.

		\begin{proposition}[Eigenvalues determine int-amplified endomorphisms: Theorem 3.3 \cite{MR4074056}]\label{thm:intampeigen}
			Let $X$ be a projective variety defined over $\Qb$. Let $f\colon X\ra X$ be a surjective endomorphism. Then $f$ is int amplified if and only if  $f^*\colon N^1(X)_\RR\ra N^1(X)_\RR$ has all eigenvalues of modulus strictly larger then 1.
		\end{proposition}
		
		From \ref{thm:intampeigen} we see that the composition of int-amplified endomorphisms are int-amplified. We naturally obtain a sub-monoid of all surjective endomorphisms. 
		
		\begin{definition}[The monoid of surjective endomorphisms]\label{def:IntSur}
			Let $X$ be a projective variety defined over $\Qb$. We let $\Send(X)$ be the monoid of surjective endomorphisms of $X$. We let $\Iamp(X)$ be the collection of int-amplified endomorphisms of $X$.
		\end{definition}
		
		The idea behind int-amplified endomorphisms is that if $\Iamp(X)\neq \emptyset$ then we have a sub-monoid $\Iamp(X)\subseteq \Send(X)$ that can be used to study $\Send(X)$. We will see that the existence of $f\in \Iamp(X)$ has consequences for the birational geometry of $X$. We now summarize the main properties of varieties with an int-amplified endomorphism.
		
		\begin{theorem}[Lemma 5.2 \cite{MR4070310}]\label{thm:propertiesofintam}
			Let $X,Y$ be a normal projective varieties. Let $f\colon X\ra X$ and $g\colon Y\ra Y$ be surjective endomorphisms.
			
			\begin{enumerate}
				\item If $\pi\colon X\ra Y$ is a surjective endomorphism and $f$ is int-amplified with $g\circ\pi=\pi\circ f$ then $g$ is int-amplified.
				\item  If $\dim X=\dim Y$ and $\pi\colon X\dashrightarrow Y$ is a dominant rational map with $g\circ\pi=\pi\circ f$ then $g$ is int-amplified.
				\item If $X$ is $\QQ$-factorial and $f$ is int amplified then $-K_X\Num E$ where $E$ is an effective $\QQ$- divisor. In particular if $\Alb(X)=0$ then $\kappa(-K_X)\geq 0$. 
			\end{enumerate}
		\end{theorem}
		
		\begin{definition}[The equivariant MMP:Meng-Zhang in \cite{MR4117085}]\label{def:equivariantMMP}
			
			Consider a sequence	of dominant rational maps
			
			\begin{align}\label{eq:equivraiantMMP}
				X_1\dashrightarrow X_2\dashrightarrow X_3\dots\dashrightarrow X_r
			\end{align}	
			
			such that each $X_i$ is a normal projective variety. Let $f=f_1\colon X_1\ra X_1$ be a surjective endomorphism. We say that \ref{eq:equivraiantMMP} is $f$-equivariant if there are surjective endomorphisms $f_i\colon X_i\ra X_i$ such that $g_i\circ f_{i+1}=f_i\circ g_i$ for all $i$, where $g_i\colon X_i\dashrightarrow X_{i+1}$ is the dominant rational mapping of \ref{eq:equivraiantMMP}.

		\end{definition}
		
		\begin{theorem}[The equivariant MMP of Meng-Zhang: Theorem \cite{MR4070310}]\label{thm:MengZhangMMP}
			Let $X$ be a $\QQ$-factorial projective variety	defined over $\Qb$ with at worst terminal singularities admitting an int-amplified endomorphism.
			
			\begin{enumerate}
				\item There are only finitely many $K_X$ negative extremal rays of $X$. Moreover if $f\colon X\ra X$ is a surjective endomorphism then there is some $n\in \ZZ_{>0}$ such that $f^n_*\colon \NE(X)_\RR\ra \NE(X)_\RR$ fixes every $K_X$ negative extremal ray.Let $R$ be any extremal $K_X$ negative extremal ray with contraction $\phi_R\colon X\ra Y_R$. Then there is a surjective endomorphism $g_R\colon Y_R\ra Y_R$ such that $g_R\circ \phi_R=\phi_R\circ f^n$. Moreover if $R$ is a flip and $\psi_R^+$ is the associated birational mapping $X\ra X_R^+$ then the induced rational mapping $f_R^+\colon X_R^+ \dashleftarrow X_R^+$ extends to a morphism $f_R^+\colon X_R^+\ra X_R^+$.
				
				\item Then for any surjective morphism $f\colon X\ra X$  there is some $n$ and a $f^n$ equivariant MMP for $f^n$ g given by \[X_1\dashrightarrow X_2\dashrightarrow X_3\dots\dashrightarrow X_r.\]
				Let $g_i\colon X_i\dashrightarrow X_{i+1}$. Then we have that.
				
				\begin{enumerate}
					\item Each $g_i$ is a contraction of a $K_X$ negative extremal ray.
					\item $X_r$ is a $Q$-Abelian variety. Note that $X_r$ might be a point. In fact there is there is a surjective endomorphism $h\colon A\ra X_r$ where $h$ is a finite and $A$ is an abelian variety. Moreover there is a surjective endomorphism $w\colon A\ra A$ such that $w\circ h=f_r\circ w$. The existence of $h$ is the definition of $Q$-Abelian, the theorem provides that the morphism $g_r$ commutes with a morphism of the covering abelian variety. In fact this holds for any surjective endomorphism of a $Q$-Abelian variety.
				\end{enumerate}	
				
				\item Let $f\colon X\ra X$ be a surjective endomorphism. Let \[X_1\dashrightarrow X_2\dashrightarrow X_3\dots\dashrightarrow X_r\] be any MMP where the $g_i\colon X_i\ra X_{i+1}$ are divisorial or fibering contractions. Then there is some $n$ such that there are surjective endomorphisms $f_i\colon X_{i}\ra X_i$ making the MMP $f$-equivariant.

			\end{enumerate}

		\end{theorem}

	\end{subsection}

	\begin{subsection}{Dynamical degrees}\label{subsec:Dynamical degrees}
		To effectively study a surjective endomorphism of $X$ we wish to assign a numerical notion of complexity of $f$ under iteration. Our such notion will be the following. 				
		
		\begin{definition}[The dynamical degree.]\label{def:dynamicaldegree}
			Let $X$ be a projective variety defined over $\Qb$. Let $f\colon X\ra X$ be a dominant morphism. The first dynamical degree of $f$ is defined to be
			\[\lambda_1(f)=\limsup_{n\ra \infty}\rho((f^{\circ n})^*\colon N^1(X)_\RR\ra N^1(X)_\RR) ).\]
			Here $\rho((f^{\circ n})^*)$ is the spectral radius of the pull back action on $N^1(X)_\RR$. 		
		\end{definition}
		
		The dynamical degree can also be defined for arbitrary dominant rational maps provided that $X$ is normal, this is to make sure that the pull back maps behave as expected. This definition is often not used in practice. We have the following.
		
		\begin{proposition}[Properties of the dynamical degree: section 1 of \cite{MR3189467} and corollary 18 in \cite{MR4080251}]\label{prop:dyndegproperties}
			Let $X$ be a projective variety defined over $\Qb$. Let $f\colon X\ra X$ be a dominant morphism. Then we have the following.
			
			\begin{enumerate}
				
				\item Let $H$ be an ample divisor. Then \[\lambda_1(f)=\lim_{n\ra \infty}((f^n)^*H\cdot H^{\textnormal{dim}X-1})^{\frac{1}{n}}.\]
				\item $\lambda_1(f)=\rho(f^*)$. That is, the dynamical degree of $f$ is the spectral radius of the action of $f^*$ on $N^1(X)_\RR.$
			\end{enumerate}
		\end{proposition}
		\end{subsection}

\end{section}

\end{section}

\begin{section}	{The tractable minimal model program}
We now describe a process for studying dynamical questions from the perspective of the minimal model program.	
\begin{definition}[Tractable minimal model programs]\label{def:mmpdef}
		Let $X$ be a $\QQ$-factorial variety with at worst terminal singularities. Suppose that $f\colon X\ra X$ is a surjective endomorphism.
		\begin{enumerate}
			\item A $f$-equivariant minimal model program or $f$-equivariant MMP is a sequence
			\begin{equation}\label{eq:tractablemmp}
				\xymatrix{X=X_0\ar@{.>}[r]^-{\phi_1}& X_1\ar@{.>}[r]^-{\phi_2}&...\ar@{.>}[r]^-{\phi_r}& X_r}
			\end{equation}
			where each $\phi_i$ is a flipping, divisorial, or fibering contraction along with \textbf{morphisms}  $f_i\colon X_i\ra X_i$ with $f_0=f$ such that $f_i\circ \phi_{i+1}=\phi_{i+1}\circ f_i$.
			\item If $X_r$ is a $Q$-abelian variety then we call the MMP \textbf{tractable}. The reason for this terminology is that if $X$ admits an int-amplified endomorphism then we can always find an MMP ending in a $Q$-abelian variety. 
			\item We call the MMP \textbf{standard} if $X_{r-1}\ra X_r$ is a fibering contraction and for $i>r-1$ we have that $X_i\ra X_{i+1}$ is birational or if $X_i\ra X_{i+1}$ is always birational, and $X_{r}$ is a minimal model.  
			\item We will often denote an $f$-equivariant MMP by $\mathcal{M}$ to denote the data of the sequence of contractions along with the conjugating morphisms. 
			\item We write $\lambda_1(f\mid_\McM)$ for the sequence whose $i^{th}$ coordinate is $(\dim X_i,\lambda_1(f_i\mid_{\phi_{i+1}}))$. The purpose of this notation is to differentiate between the various types of minimal model operations that may occur.
			\item If $X_r$ is not zero dimensional then we say that $f_r\colon X_r\ra X_r$ is a \textbf{primordial model} of $f$ and call $\lambda_1(f_r)$ the \textbf{primordial dynamical degree} of $\McM$. If $X_r$ is zero dimensional we call $f_{r-1}\colon X_{r-1}\ra X_{r-1}$ a \textbf{primordial model} of $f$ and  $\lambda_1 (f_{r-1})$ the \textbf{primordial dynamical degree} of $\McM$. We denote the primordial degree of $\McM$ to be $\lambda_1^{\textnormal{pr}}(\McM)$.
			\item We define the \textbf{primordial dynamical degrees} of the morphism $f$ as \begin{align}
				\underline{\lambda}_1^{\textnormal{pr}}(f)&=\min\{\lambda_1^{\textnormal{pr}}(\McM):\McM\textnormal{ a tractable }f\textnormal{ equivariant MMP} \}\\
				\overline{\lambda}_1^{\textnormal{pr}}(f)&=\max\{\lambda_1^{\textnormal{pr}}(\McM):\McM\textnormal{ a tractable }f\textnormal{ equivariant MMP} \}
			\end{align}
			if a tractable $f$-equivariant MMP exists for $f$ and $\infty$ otherwise. We think of the collection of primordial models of $f$ as its collection of ancestors. The number $	\underline{\lambda}_1^{\textnormal{pr}}(f)$ measures the simplest ancestor, while $\overline{\lambda}_1^{\textnormal{pr}}(f)$ measures the most complex ancestor. 
		\end{enumerate}
		
\end{definition}
Our goal is to capture the complexity of those morphisms which are built from $Q$-abelian varieties. If $\McM$ is an $f$-equivariant MMP as in \ref{eq:tractablemmp} and $X_r$ is zero dimensional. Then a primordial model for $f$ is $X_{r-1}$ which must have Picard number $1$. If $X_{r}$ is positive dimensional then $f$ is built out of a surjective endomorphism of a $Q$-abelian variety, for example an abelian variety if there is a tractable MMP. 
	\begin{example}
	\normalfont	Let $f_1\times f_2\colon \PP^{n_1}\times \PP^{n_2}\ra  \PP^{n_1}\times \PP^{n_2}$ be a surjective morphism.  Since  $\PP^{n_1}\times \PP^{n_2}$ admits two extremal contractions, $\pi_i\colon \PP^{n_1}\times \PP^{n_2}\ra \PP^{n_i}$ there are two $f_1\times f_2$ equivariant minimal model programs
		\[\xymatrix{\PP^{n_1}\times\PP^{n_2}\ar[d]_{\pi_i}\ar[r]^{f_1\times f_2}& \PP^{n_1}\times\PP^{n_2}\ar[d]^{\pi_i}\\ \PP^{n_i}\ar[r]_{f_i}& \PP^{n_i}}\] 
		In this case we have $\overline{\lambda}_1^{\textnormal{pr}}(f_1\times f_2)=\max\{\deg f_1,\deg f_2\}$ and $\underline{\lambda}_1^{\textnormal{pr}}(f_1\times f_2)=\min\{\deg f_1,\deg f_2\}$.	\end{example}
	\begin{example}
	\normalfont	Let $X$ be a smooth variety admitting an MMP
		\[\xymatrix{X=X_0\ar@{.>}[r]^{\phi_1}& X_1\ar@{.>}[r]^{\phi_2}&...\ar@{.>}[r]^{\phi_r}& X_r}\]
		with $X_r$ a $Q$-abelian variety. Let $f_i=\textnormal{id}_{X_i}$. Then this is a equivariant MMP for $\textnormal{id}_X$ and $\overline{\lambda}_1(\textnormal{id}_X)=\underline{\lambda}_1(\textnormal{id}_X)=1$
	\end{example}
	
	\begin{example}
	\normalfont	Let $X$ be a simple abelian variety of Picard number $1$ and $\tau_c\colon X\ra X$ a translation by a non-torsion point. Then $\textnormal{id}\colon X\ra X$ is the only equivariant MMP for $\tau_c$. Thus $\overline{\lambda}_1(\tau_c)=\underline{\lambda}_1(\tau_c)=1$.
	\end{example}
	\begin{tracmmp}\label{tractmmp}
		\normalfont Let $X$ be a variety defined over $\Qb$ with mild singularities so that some version of the minimal model program is possible. Suppose $f\colon X\ra X$ is a surjective endomorphism . Consider some dynamical property of surjective endomorphisms property $\mathcal{D}$. Our goal is to check if $f$ has $\mathcal{D}$. For example, $\mathcal{D}$ could be if $f$ satisfies the Kawaguchi-Silverman conjecture, $f$ has arithmetic eigenvalues, or if $f$ has a dense set of pre-periodic points. Then we have the following program.
		\begin{enumerate}
			\item Verify that if we have a diagram
			\[\xymatrix{X\ar[r]^f\ar[d]_\phi & X
				\ar[d]^\phi\\ Y\ar[r]_g & Y }\]
			with $g$ surjective and $\phi$ a divisorial contraction then $\mathcal{D}$ holds for $f$ if and only if it holds for $g$.
			\item Let $\phi\colon X\ra Y$ is a flipping contraction with flip $\phi^+\colon X^+\ra Y$. If $f^+\colon X^+\dashrightarrow X^+$ extends to a morphism, verify that $f$ has $\mathcal{D}$ if and only if $f^+$ has $\mathcal{D}$.
			\item Determine a condition $F(\mathcal{D})$ such that if
			\[\xymatrix{X\ar[r]^f\ar[d]_\phi & X
				\ar[d]^\phi\\ Y\ar[r]_g & Y }\]
			is a diagram with $\phi$ fibering and that $\phi$ has $F(\mathcal{D})$ then $f$ has $\mathcal{D}$ if and only if $g$ has $\mathcal{D}$. We think of $F(\mathcal{D})$ of some formal notion that says that $\phi$ has well behaved fibers.
			\item Verify that all primordial models have $\mathcal{D}$. In other words, show that surjective endomorphisms of $Q$-abelian varieties and surjective endomorphisms of Picard number 1 varieties have $\mathcal{D}$. 
			\item Define that a tractable MMP $\McM$ has $F(\mathcal{D})$ if every fibering contraction in $\McM$ has $F(\mathcal{D})$. 
			\item Conclude that all surjective endomorphisms that possess a tractable MMP with $F(\mathcal{D})$ has $\mathcal{D}$. 
		\end{enumerate}
		
	\end{tracmmp}
	We now illustrate this idea with some examples of the program and some variants.
\end{section}	
	
	\begin{section}{Pre-periodic points for varieties admitting an int-amplified endomorphism}\label{sec:preperiodicpoints}

In this section we begin to enact the tractable minimal model program outlined in \ref{tractmmp} with $\mathcal{D}$ being the property that a surjective endomorphism has a dense set of pre-periodic points. We first handle (1) and (2) in \ref{tractmmp}.
\begin{proposition}
	Let $X$ be a variety defined over $\overline{\QQ}$ and let $f\colon X\ra X$ be a surjective endomorphism. Fix $n\geq 1$. Then $f$ has a dense set of pre-periodic points if $f^n$ does. 
\end{proposition}
\begin{proof}
	Let $f^n$ have a dense set of pre-periodic points. Let $U$ be an open set of $X$. Then there is a point $u\in U$ with $f^{an}(u)=f^{nb}(u)$. Then $u$ is a pre-periodic point for $f$ as well. 
\end{proof}

\begin{proposition}\label{prop:morper1}
	Let $X,Y$ be a irreducible varieties defined over $\overline{\QQ}$. Let $\phi\colon X\ra Y$ be a birational morphism. Let $f\colon X\ra X$ and $g\colon Y\ra Y$ be a surjective endomorphisms. Suppose that $\phi\circ f=g\circ \phi$. Then $f$ has a dense set of pre-periodic points if and only if $g$ has a dense set of pre-periodic points 
\end{proposition}
\begin{proof}
	Let $U\subseteq X$ be an open set of $X$ and $V\subseteq Y$ be an open set of $Y$ with $\phi\colon U\ra V$ an isomorphism. Let $\McP_f$ be the set of pre-periodic points of $f$ and $\McP_g$ the set of pre-periodic points of $g$. Suppose that $\McP_f$ is dense in $X$. Let $W$ be an open set in $Y$. Then $W\cap V$ is non-empty and open and $\phi^{-1}(W\cap V)$ is open in $X$. So there is a point $p\in \phi^{-1}(W\cap V)$ such that $f^n(p)=f^k(p)$. Then $\phi(f^n(p))=g^n(\phi(p))$ and $\phi(f^k(p))=g^k(\phi(p))$ which tells us that $q=\phi(p)$ is a pre-periodic in $W\cap V$. So $\McP_g$ is dense in $Y$. Now suppose that $\McP_g$ is dense in $Y$. Let $W$ be an open set in $X$. Then $\phi(W\cap U)$ is an open set of $V$ and so there is a point $q\in\phi(W\cap U)$ with $g^n(q)=g^k(q)$. Since $q=\phi(p)$ for some $p\in U\cap W$ we have $\phi(f^n(p))=\phi(f^k(p))$. Since $\phi$ is an isomorphism on $U$ we have that $f^n(p)=f^k(p)$ and so $\McP_f$ is dense in $X$ as needed. \end{proof}

We immediately obtain the following results.

\begin{corollary}\label{cor:toricdensediv}
	Let $X$ be a $\QQ$-factorial variety with at worst terminal singularities. Let $f\colon X\ra X$ be a surjective endomorphism. Let $\phi\colon X\ra Y$ be a divisorial contraction. Let $g\colon Y\ra Y$ be a surjective endomorphism. Suppose that $\phi\circ f=g\circ \phi$. Then $f$ has a dense set of pre-periodic points if and only if $g$ has a dense set of pre-periodic points. 
\end{corollary}
\begin{corollary}\label{cor:toricdenseflip}
	Let $X$ be a $\QQ$-factorial variety with at worst terminal singularities. Let $f\colon X\ra X$ be a surjective endomorphism. Let $\phi\colon X\ra Y$ be a flipping contraction.  Let $\psi\colon X\dashrightarrow X^+$ be the associated flipping birational mapping. Let $f^+\colon X^+\ra X^+$ be a surjective endomorphism. Suppose that $\psi\circ f=f^+\circ \psi$. Then $f$ has a dense set of pre-periodic points if and only if $f^+$ has a dense set of pre-periodic points.
\end{corollary}
\begin{proof}
	After iterating $f$ we may assume by \cite[Theorem 3.3]{1908.11537} that we have a diagram \[\xymatrix{X\ar[r]^f\ar[d]_\phi & X\ar[d]^\phi \\ Y\ar[r]_g & Y\\ X^+\ar[u]^{\phi^+}\ar[r]_{f^+}& X^+\ar[u]_{\phi^+}}\]
	If $f$ has a dense set of pre-periodic points then so does $g$ by Proposition \ref{prop:morper1} and therefore so does $f^+$ by Proposition \ref{prop:morper1} applied once more. The same argument applies if $f^+$ has a dense set of pre-periodic points. 
\end{proof}
We now define $F(\mathcal{D})$. Recall that here $\mathcal{D}$ is the property that a morphism has a dense set of pre-periodic points. 
\begin{definition}\label{def:enoughpreperiodicpoints}
		Suppose we are given a commuting diagram
		\[\xymatrix{X\ar[r]^f\ar[d]_\phi & X\ar[d]^\phi\\ Y\ar[r]_g & Y}\]
		where $X,Y$ are normal projective varieties and $f,g,\phi$ are surjective endomorphisms. Suppose further that $\phi$ has connected fibers and that the general fiber is normal. We say that $f$ has enough pre-periodic points (with respect to $\phi$) if there is a non-empty open set $W\subseteq Y$ such that for all $p\in W$ with $g^n(p)=g^{n+k}(p)$ for some $n,k\in \ZZ_{\geq 0}$ we have:
		\begin{enumerate}
			\item The induced morphism
			\[f^k\colon \phi^{-1}(g^n(p))\ra \phi^{-1}(g^{n+k}(p))=\phi^{-1}(g^n(p))\]
			has a dense set of pre-periodic points and $\phi^{-1}(g^n(p))$ is normal.
			\item $f^n\colon \phi^{-1}(p)\ra \phi^{-1}(g^n(p))$ is dominant.
		\end{enumerate} 
		\end{definition}
	\begin{example}
		\normalfont Let $f\colon X\ra X$ be a surjective endomorphism defined over a field $K$. Suppose that $X$ is a normal projective variety. Then we have
		\[\xymatrix{X\ar[rr]_f \ar[dr]&& X\ar[dl]\\& \spec k&}\] 
		Then $f$ has enough pre-periodic points for the structure morphism if and only if $f$ has a dense set of pre-periodic points. 
	\end{example}
	\begin{example}
	\normalfont	Let $\pi\colon \PP\McE\ra X$ be the structure morphism of a projective bundle. Suppose we have a diagram \[\xymatrix{\PP\McE\ar[r]^f \ar[d]_\pi& \PP\McE\ar[d]^\pi\\X\ar[r]_g& X}\] Suppose that $\lambda_1(f\mid_\pi)>1$. Then $f$ has enough pre-periodic points with respect to $\pi$. This is because for $p\in X$ we have an induced morphism
		\[f\colon \pi^{-1}(p)\ra \pi^{-1}(g(p)).\] 
		The degree of $f$ on the fibers is $\lambda_1(f\mid_\pi)>1$. So $f$ restricted to a fiber is a polarized endomorphism of projective space; any polarized morphism of projective space has a dense set of pre-periodic points by \cite[5.3]{MR1995861}.
	\end{example}
We now formalize how these ideas relate to the minimal model program. 
	
	An $f$-equivariant MMP $\McM$ \textbf{has enough pre-periodic points} if for all $\phi_i\colon X_{i-1}\ra X_i$ of fibering type in $\McM$ we have that $f_{i-1}$ has enough pre-periodic points with respect to $\phi_i$.
	\begin{definition}\label{def:mmpdefpre}
		Let $X$ be a $\QQ$-factorial variety with at worst terminal singularities. Suppose that $f\colon X\ra X$ is a surjective endomorphism. Consider a tractable $f$-equivariant MMP $\McM$ given by
		\[\xymatrix{X=X_0\ar@{.>}[r]^{\phi_1}& X_1\ar@{.>}[r]^{\phi_2}&...\ar@{.>}[r]^{\phi_r}& X_r}.\]
		We say that	$\McM$ has enough pre-periodic points if for all $\phi_i\colon X_{i-1}\ra X_i$ of fibering type in $\McM$ we have that $f_{i-1}$ has enough pre-periodic points with respect to $\phi_i$.
	\end{definition}
	Our strategy is to run an equivariant MMP on $f$ to determine if $f$ has a dense set of pre-periodic points. However, a basic issue with the above approach is morphisms with $\overline{\lambda}_1^{\textnormal{pr}}(f)=1$. These are morphisms whose primordial ancestors all have dynamical degree $1$. In other words their simplest ancestors may be akin to a translation $\tau_c\colon A\ra A$ or a non-trivial isomorphism $f\colon \PP^n\ra \PP^n$. Such a morphism is induced by a morphism of dynamical degree 1, which may not have a dense set of pre-periodic points. More generally using the notation of Definition \ref{def:mmpdef} if $\phi_{i+1}\colon X_i\ra X_{{i+1}}$ is a fibering type contraction and the general fiber of $\phi_{i+1}$ is a Fano variety of Picard number $1$, then if $\lambda_1(f_i\mid_{\phi_{i+1}})=1$ we have that $f_i$ may fail to have enough pre-periodic points with respect to $\phi_{i+1}$. The definitions given above are meant to isolate precisely where these issues may arise when applying an MMP. Given an MMP $\mathcal{M}$ as in Definition \ref{def:mmpdef} we see that the coordinates $i$ of $\lambda_1(f\mid_{\McM})$ are of the form $(\dim X_i,1)$ with $\dim X_{i+1}<\dim X_i$ are where problems occur. Thus we see that the crucial factor to understand when to the denseness of pre-periodic points of endomorphisms is the behavior with respect to fibering type contractions.
	\begin{proposition}\label{cor:toricdensefib}
		Let $\phi\colon X\ra Y$ be a fibering type extremal contraction of a $\QQ$-factorial variety with at worst terminal singularities. Suppose that we have a diagram
		\[\xymatrix{X\ar[r]^f\ar[d]_\phi & X\ar[d]^\phi\\ Y\ar[r]_g & Y}\]
		with $f,g$ surjective morphisms. Suppose that $g$ has a dense set of pre-periodic points and that $f$ has enough pre-periodic points with respect to $\phi$. Then $f$ has a dense set of pre-periodic points. 
	\end{proposition}
	\begin{proof}
		Let $U$ be an open set in $X$.  Set $W_0=\phi(U)$. Since $\phi$ is a surjective morphism between normal varieties \ref{lem:suropen} says that $\phi$ is open. Consequently we have that $W_0$ is open. Let $W$ be as in \ref{def:enoughpreperiodicpoints}. Set $W_0^\prime=W\cap W_0$. We may find $p\in W_0^\prime$ with $g^n(p)=g^m(p)$ since $g$ has a dense set of pre-periodic points. We may further take $p$ general so that $\phi^{-1}(p)$ is normal and connected; the general fiber of a Mori-fiber space is normal and connected. Suppose that $n\leq m$ and that $n+k=m$. By the definition of enough pre-periodic points $\phi^{-1}(g^n(p))$ is normal and by our choice of $p$ we have that $\phi^{-1}(p)$ is normal. Both being connected and normal they are irreducible and $f^n\colon \phi^{-1}(p)\ra \phi^{-1}(g^n(p))$ is the composition $f\circ \iota\colon \phi^{-1}(p)\ra \phi^{-1}(g^n(p)) $ where $\iota\colon \phi^{-1}(p)\ra X$ is the closed immersion. Since $f^n$ is finite and closed immersions are finite we have that $f^n$ is finite unto its image. Thus $f^n\colon \phi^{-1}(p)\ra \phi^{-1}(g^n(p))$ is a dominant finite morphism with normal target. It follows that 
		\[f^n\colon \phi^{-1}(p)\ra \phi^{-1}(g^n(p))\] is an open mapping by \ref{lem:suropen}. Taking $U^\prime=U\cap \phi^{-1}(p)$ we have that $f^n(U^\prime)$ is open in $\phi^{-1}(g^n(p))$. Since $f$ has enough pre-periodic points with respect to $\phi$, the pre-periodic points of 
		\[f^k\colon \phi^{-1}(g^n(p))\ra\phi^{-1}(g^{n+k}(p))=\phi^{-1}(g^n(p))\]
		are dense by definition. So there is a point $q\in f^n(U^\pp) $ with $f^{lk}(q)=f^{lk+t}$ for some $t>0$. Since $q\in f^n(U^\pp)$ we have that $q=f^n(a)$ for some $a\in U^\prime$. Thus
		\[f^{kl+n}(a)=f^{kl}(f^n(a))=f^{kl+t}(f^n(a))=f^{kl+n+t}(a)\]
		so that $a$ is a pre-periodic point of $U^\prime$. Since $U^\prime\subseteq U$ we have that the pre-periodic points of $f$ are dense as claimed. 
		\end{proof}
	\begin{theorem}\label{thm:intamplifiedpreper1}
		Let $X$ be a $\QQ$-factorial variety with at worst terminal singularities. Suppose further that $X$ is rationally connected with $\Alb(X)=0$. Let $f\colon X\ra X$ be a surjective endomorphism. Suppose that $f$ has a tractable MMP with enough pre-periodic points. Then $f$ has a dense set of pre-periodic points.
	\end{theorem}
	\begin{proof}
		We induct on the Picard number $\rho=\rho(X)$. If $\rho(X)=1$ then $f^n\colon X\ra X$ are the only possible candidates for a MMP. By assumption for some $n$ we have that $f^n\colon X\ra X$ has a dense set of pre-periodic points. Thus $f$ does as well. Now let $\rho(X)>1$. After iterating $f$ we have by assumption the existence of a tractable MMP with enough pre-periodic points. If a flipping operation appears in the MMP then we have a diagram
		\[\xymatrix{X_i\ar[r]^{f_i}\ar@{.>}[d]_{\phi_{i+1}} & X_i\ar@{.>}[d]^{\phi_i+1}\\ X_{i+1}\ar[r]_{f_{i+1}} & X_{i+1}}\]
		Then by Corollary \ref{cor:toricdenseflip} $f_i$ has a dense set of pre-periodic points if and only if $f_{i+1}$ does. Since by assumption we have a finite MMP, we eventually hit a divisorial or fibering contraction. Suppose that
		\[\xymatrix{X\ar[r]^f\ar[d]_\phi & X\ar[d]^\phi\\ Y\ar[r]_g & Y}\]
		is the first non-flipping contraction in the MMP. Suppose first that $\phi$ is divisorial. Then $g\colon Y\ra Y$ has a MMP with enough pre-periodic points by construction and by induction we have that $g$ has a dense set of pre-periodic points. By Corollary \ref{cor:toricdensediv} $f$ does as well. On the other hand if $\phi$ is fibering then as above by induction we have that $g$ has a dense set of pre-periodic points. Since by assumption $f$ has enough pre-periodic points for $\phi$ we apply Corollary \ref{cor:toricdensefib}) and obtain that $f$ has a dense set of pre-periodic points as needed.  
	\end{proof}
As an application of the above ideas we analyze the behavior of the pre-periodic points of toric morphisms between $\QQ$-factorial toric varieties from this perspective. 
\begin{proposition}\label{prop:pic1toric}
		Let $X$ be a $n$-dimensional $\QQ$-factorial projective toric variety with Picard number $1$. Let $f\colon X\ra X$ be a surjective toric morphism. Then some iterate of $f$ is polarized or $f$ is the identity.
\end{proposition}
	\begin{proof}
		Since $X_\Sigma$ is of Picard number one, every equivariant surjective endomorphism is induced by a dilation. If $f$ is induced by $\phi\colon N\ra N$ and $\phi(v)=nv$ and $n>1$ then $f$ is polarized and we are done by \cite[theorem 5.1]{MR1995861} Otherwise $n=1$ and we are done.
		\end{proof}
\begin{proposition}\label{prop:toricblock}
		Let $X$ be a $\QQ$-factorial toric variety and $f\colon X\ra X$ a surjective toric morphism. Let $\phi\colon X\ra Y$ be a fibering type extremal contraction and let $g\colon Y\ra Y$ be a toric morphism with $\phi\circ f=g\circ\phi$.  Let $T_Y$ be the dense torus of $Y$ and $T_X$ the dense torus of $X$. Then for $t\in T_Y$ we have that $\phi^{-1}(t)\cong X^\prime$ for some $\QQ$-factorial projective toric variety of Picard number 1. Moreover if $g(t)=t^\prime$ for $t\in T_Y$ then the induced map $f\colon \phi^{-1}(t)\ra \phi^{-1}(t^\prime)$ is a surjective toric morphism.    
	\end{proposition}
	\begin{proof}
		Here we follow \cite[3.2]{1204.3883}.	
		Let $X=X_\Sigma$ for a fan $\Sigma\subseteq N_\RR$. The fibering type contraction can be given by a mapping of lattices \[\xymatrix{0\ar[r]& N_0\ar[r]^i\ar[r]& N\ar[r]^\phi& N^\prime\ar[r]& 0}\]
		where $\phi$ is the natural quotient mapping and $Y=Y_{\Sigma^\prime}$ where $\Sigma^\prime$ is the quotient fan. Let $\Sigma_0=\{\sigma\in \Sigma:\sigma\subseteq (N_0)_\RR\}$. Then one can check (see for example \cite[3.2]{1204.3883}) that $\phi^{-1}(T_{N^\prime})\cong T_{N^\prime}\times X_{N_0,\Sigma_0}$. Thus the fibers above the torus are naturally toric varieties with the torus $N_0$ inherited from the torus action on $X_\Sigma$. Now let $t\in T_{N^\prime}$ and consider the morphism $f\colon \phi^{-1}(t)\ra \phi^{-1}(g(t))$. Since $f$ is equivariant and by construction induces a map on $Y_\Sigma$ we have that $f$ preserves $N_0$ and thus for $s\in T_{N_0}$ we have that $f(s)\in T_{N_0}$. Since $f$ is equivariant we obtain that $f\colon \phi^{-1}(t)\ra \phi^{-1}(g(t))$ is as well with respect to the action of $T_{N_0}$ and thus the action on the fibers is equivariant. 
	\end{proof}
	
	\begin{corollary}\label{cor:enoughtoricpoints}
		Let $X$ be a $\QQ$-factorial toric variety and $f\colon X\ra X$ a surjective toric morphism. Let $\phi\colon X\ra Y$ be a fibering type extremal contraction. Then we have a commuting diagram \[\xymatrix{X\ar[r]^f\ar[d]_\phi & X\ar[d]^\phi\\ Y\ar[r]_g & Y}\] where $g$ is toric. Furthermore $f$ has enough pre-periodic points for $\phi$.	
	\end{corollary}
	\begin{proof}
		Take $W$ to be $T_Y$. Then by Proposition \ref{prop:toricblock} we have that $f$ is toric on the fibers which have Picard number 1 and $f$ has a dense set of pre-periodic points by Proposition \ref{prop:pic1toric}. 
	\end{proof}

	\begin{theorem}\label{thm:toricpreper1}
		Let $X$ be a $\QQ$-factorial projective toric variety defined over $\overline{\QQ}$. Let $f\colon X\ra X$ be a surjective toric morphism. Then the set of pre-periodic points of $f$ is Zariski dense. 
	\end{theorem}
	\begin{proof}
		Using the toric minimal model program we may choose equivariant MMP \[\xymatrix{X=X_0\ar@{.>}[r]^{\phi_1}& X_1\ar@{.>}[r]^{\phi_2}&...\ar@{.>}[r]^{\phi_r}& X_r=\textnormal{pt}}\]
		and we may take all morphisms here to be toric. Furthermore, by Corollary \ref{cor:enoughtoricpoints} every fibering type contraction has enough pre-periodic points. So by Theorem \ref{thm:intamplifiedpreper1} we have that $f$ has a dense set of pre-periodic points as needed. 
		\end{proof}
\end{section}
	
\begin{section}{Medvedev-Scanlon conjecture}\label{sec:MSC}
		
		In this section we illustrate how to apply the tractable minimal model program outlined in  definition \ref{tractmmp} to the Medvedev-Scanlon conjecture. In other words we take $\mathcal{D}$ to be the property that $f\colon X\ra X$ has a point with a dense orbit. 	
		\begin{definition}
			Let $X$ be a projective variety and suppose that $f\colon X\ra X$ is a surjective endomorphism. We say that $f$ is \emph{fiber preserving} if there is a positive dimensional variety $Z$ and a dominant rational map $\psi\colon X\dashrightarrow Z$ such that $\psi\circ f=\psi$. 
		\end{definition}
		The conjecture is usually stated as follows.
		\begin{conjecture}[The Medvedev-Scanlon conjecture]\label{conj:M-S}
			Let $X$ be an irreducible variety defined over an algebraically closed field $F$ of characteristic 0. Let $\phi\colon X\dashrightarrow X$ be a dominant rational map. If $\phi$ is not fiber preserving then there is a point $x\in X(F)$ with a forward dense orbit under $\phi$.
		\end{conjecture}
		
The Medvedev-Scanlon conjecture behaves well with respect to iteration.
		\begin{lemma}[Lemma 2.1 \cite{MR3711275}]\label{conj:weakM-S}
			Let $X$ be an irreducible variety defined over an algebraically closed field $F$ of characteristic 0. Let $\phi\colon X\dashrightarrow X$ be a dominant rational map. If $\phi^n$ is not fiber preserving for some $n\geq 1$ then $\phi$ is not fiber preserving. 
		\end{lemma}
		\begin{definition}\label{def:difficulty}
			Let $X$ be a normal projective variety defined over $\Qb$. Let $f\colon X\ra X$ be a dominant morphism. We say that a closed sub-variety $V\subseteq X$ is \emph{dynamically constructible} or D-constructible by $f$ if there is a closed irreducible sub-variety $W\subseteq X$ such that
			\begin{equation}
				\overline{\OO_f(W)}=V.   
			\end{equation}
			We say that $W$ a generator for $V$. If $V$ is a dynamically constructible sub-variety of $X$ we say the \emph{dynamical difficulty} or D-difficulty of $V$ is the number
			\begin{equation}
				\textnormal{Difficulty}(V,f)=\min_{W\subseteq X, W \textnormal{ a generator for }V}\dim(W).    
			\end{equation}
			We think of the difficulty as a measure of how hard it is to dynamically construct $V$. If $V$ is not D-constructible then we set the difficulty to $\infty$. If the $D$-difficulty of $V$ is finite then we say that an irreducible sub-variety $W$ is a \textbf{progenitor} of  $V$ if $\overline{\OO_f(W)}=V$ and $\dim W=\textnormal{Difficulty}(V,f)$. In other words, the progenitors of $V$ are the generators of minimal dimension. 
		\end{definition}
		\begin{example}
			Let $X$ be any projective variety and $f$ a finite order automorphism of $X$. Then the D-difficulty of $X$ with respect to $f$ is $n=\dim X$.
			\end{example}
		\begin{example}
			Let $X=C\times C$ where $C$ is an elliptic curve. Let $f(x,y)=(2x,y)$. Then $f$ has no point with dense forward orbit but if $P$ is a non-torsion point of $C$ then $P\times C$ is a curve with a dense orbit under $f$. So the D-difficulty of $C\times C$ is $1$. 
		\end{example}
		With this notation we have the following rephrasing of the Kawaguchi-Silverman conjecture. Let $X$ be a normal projective variety and $f\colon X\ra X$ a surjective endomorphism. Suppose that $\textnormal{Difficulty}(X,f)=0$. Then if $P$ is an progenitor for $X$ we have that $\alpha_f(P)=\lambda_1(f)$. Our idea here is to point out that the Kawaguchi-Silverman conjecture is only interesting when $X$ can be built as an orbit closure of the forward orbit of a point of $f$. Now let $X$ be a normal projective variety defined over $\Qb$ equipped with a surjective endomorphism $f\colon X\ra X$. The Medvedev-Scanlon conjecture is equivalent to the statement that the $D$-difficulty of $X$ with respect to $f$ is zero unless $f$ preserves a rational fibration.
		\begin{definition}[Relative difficulty]
			Let $X,Y$ be normal irreducible projective varieties. Let $f\colon X\ra X$ be a surjective endomorphism and $\pi\colon X\ra Y$ a surjective endomorphism with $g\circ\pi =\pi\circ f$. We say that $f$ has relative D-difficulty at most $k$ with respect to $\pi$ if there is a non-empty open set $W\subseteq Y$ such that for all $y\in W(\Qb)$ we have that
			\[V_y=\overline{\OO_f(\pi^{-1}(y))}\]
			has $D$-difficulty at most $k$. In other words for all $y\in W$ we have
			\begin{equation}
				\textnormal{Difficulty}(\OO_f(\pi^{-1}(y)),f)\leq k.
			\end{equation}
			We say that the relative difficulty is $k$ if the relatively difficulty is at most $k$ and \textbf{not} at most $k-1$.
		\end{definition} 
		
	\begin{definition}[Dynamical set up for the Medvedev-Scanlon conjecture.]\label{def:mmpdefMS}
			Let $X$ be a $\QQ$-factorial normal variety with at worst terminal singularities. Suppose that $f\colon X\ra X$ is a surjective endomorphism. Consider an $f$-equivariant MMP
			\[\xymatrix{X=X_0\ar@{.>}[r]^{\phi_1}& X_1\ar@{.>}[r]^{\phi_2}&...\ar@{.>}[r]^{\phi_r}& X_r}\]
			where each $\phi_i$ is a flipping, divisorial, or fibering contraction along with morphisms $f_i\colon X_i\ra X_i$ with $f_0=f$ such that $f_i\circ \phi_{i+1}=\phi_{i+1}\circ f_i$. An $f$-equivariant MMP $\McM$ has relative difficulty at most $k$ if for all $\phi_i\colon X_{i-1}\ra X_i$ of fibering type in $\McM$  we have that $f_{i-1}$ has relative difficulty at most $k$ with respect to $\phi_i$.
			\end{definition}
\begin{lemma}\label{lem:MCbirkey}
		Let $f\colon X\ra X$ be a surjective endomorphism of a normal $\QQ$-factorial normal variety with at worst terminal singularities. Let $\phi\colon X\ra Y$ be a contraction morphism and $g\colon Y\ra Y$ a surjective endomorphism with $g$ surjective and $g\circ \phi=\phi\circ f$. Suppose that $\phi$ is birational. Then $\textnormal{Difficulty}(X,f)=r\iff\textnormal{Difficulty}(Y,g)=r$. In particular, $f$ has a point with dense orbit if and only if $g$ does. 
		\end{lemma}
	\begin{proof}
		Let $E$ be the exceptional locus of $\phi$ and let $Z=\phi(E)$. Suppose that $W\subseteq X$ is closed and irreducible and $\OO_f(W)$ is dense in $X$.  I claim that $\phi(W)$ is dense in $Y$ with respect to $g$. Let $V$ be open in $Y$. Then $\phi^{-1}(V)=U$ is open in $X$. So there is some $n\geq 1$ with $f^n(W)\cap U\neq \emptyset$ as $\OO_f(W)$ is dense in $X$. So there is some $P\in W$ with $f^n(P)\in U$. Then $g^n(\phi(P))\in V$. Since $\phi(P)\in \phi(W)$ we have that $\OO_g(\phi(W))$ is dense in in $Y$. Conversely, suppose that $W\subseteq Y$ is closed and irreducible and $\OO_g(W)$ is dense in $Y$. I claim that $\OO_f(\phi^{-1}(W))$ is dense in $X$. Let $U$ be open in $X$. Then $\phi(U\setminus E)=V$ is open in $Y\setminus Z$.  Since $W$ is dense in $Y$ we have that $g^n(W)\cap V\neq \emptyset$. Therefore there is a point $Q\in W$ with $g^n(Q)\in V$. Let $P\in \phi^{-1}(Q)$. Then $f^n(P)\in \phi^{-1}(g^n(P))$. Since $g^n(P)\notin Z$ and $\phi$ is birational $\phi^{-1}(g^n(P))$ is a singleton (because $Z$ is the image of the exceptional locus of $\phi$). Thus $f^n(P)$ is the unique point of $X$ with $\phi(f^n(P))=g^n(Q)$. Since $g^n(Q)\in V$ we have that the point in the fiber $\phi^{-1}(g^n(Q))$ must be in $U$. Therefore $f^n(P)\in U$ and so $\phi^{-1}(W)$ is dense in $X$ as needed. Now let $W$ be a progenitor of $X$ as in \ref{def:difficulty}. Then $\OO_g(\phi(W))$ is dense in $Y$ by the above argument. Since $\OO_f(W)$ is dense we have that $W$ is not contained in the exceptional locus of $f$. Thus $\dim \phi(W)=\dim W$ as $\phi$ is birational and $W$ is not contained in the exceptional locus. Thus $\textnormal{Difficulty}(X,f)\geq \textnormal{Difficulty}(Y,g)$ as $Y$ has a generator of dimension $\textnormal{Difficulty}(X,f)$. Conversely, let $W$ be a progenitor of $Y$. Then $W$ is not contained in $Z=\phi(E)$ the image of the exceptional locus. As $\phi$ is birational and $W$ is not contained in the image of exceptional locus we have $\dim W=\dim \phi^{-1}(W)$. By the above argument we have that $\phi^{-1}(W)$ has dense orbit. Thus some irreducible component of $\phi^{-1}(W)$ has a dense orbit and we obtain that $\textnormal{Difficulty}(X,f)\leq \textnormal{Difficulty}(Y,g)$ as $X$ has a generator of dimension $\textnormal{Difficulty}(Y,g)$. We thus obtain the desired result that $\textnormal{Difficulty}(X,f)= \textnormal{Difficulty}(Y,g)$. 
\end{proof}

We see that the difficulty is preserved by a conjugating birational morphism. 
		\begin{corollary}\label{cor:MCbirkey}
			Let $f\colon X\ra X$ be a surjective endomorphism of a normal $\QQ$-factorial variety with at worst terminal singularities. Let $\phi\colon X\ra Y$ be a flipping contraction and $g\colon Y\ra Y$ a surjective endomorphism with $g$ surjective and $g\circ \phi=\phi\circ f$. Suppose that $\psi\colon X\dashrightarrow X^+$ is the canonical birational morphism to the flip of $\phi$ and that the birational mapping $f^+=\phi \circ f\circ \phi^{-1}\colon X^+\ra X^+$ extends to a surjective morphism $f^+X^+\ra X^+$. Then $\textnormal{Difficulty}(X,f)=r\iff \textnormal{Difficulty}(X^+,f^+)=r$. In particular $f$ has a point with Zariski dense orbit if and only if $f^+$ does.
		\end{corollary}
		
\begin{lemma}\label{lem:MSfibrationbackbone}
			Let $f\colon X\ra X$ be a surjective endomorphism of a normal $\QQ$-factorial variety with at worst terminal singularities. Let $\phi\colon X\ra Y$ be a contraction morphism and $g\colon Y\ra Y$ a surjective endomorphism with $g$ surjective and $g\circ \phi=\phi\circ f$. Suppose that $\phi$ is of fibering type. Let $W\subseteq Y$ be closed. Then $\OO_g(W)$ is dense in $Y$ if and only if $\OO_f(\phi^{-1}(W))$ is dense in $X$. 	
		\end{lemma}
		\begin{proof}
			Suppose that $W$ has a dense orbit under $g$. Let $U$ be a non-empty open set of $X$ and set $V=\phi(U)$. Then $V$ is open in $Y$ by \ref{lem:suropen}. So $g^n(W)\cap V$ is non-empty as $W$ has a dense orbit. This intersection contains general points, so we may find $P\in W$ with $g^n(P)\in V$ and $\phi^{-1}(g^n(p))$ a normal $\QQ$-factorial Fano variety of dimension $\dim X-\dim Y$. We can then find $Q^\pp\in U$ with $\phi(Q^\pp)=g^n(P)$. Now consider the mapping on fibers
			\begin{equation}\label{eq:diffeq1}
				h=f^n\colon \phi^{-1}(P)\ra \phi^{-1}(g^n(P)).
				\end{equation}
			The mapping $h$ is open since $f^n$ is open by \ref{lem:suropen}. Therefore $h$ is dominant since $\phi^{-1}(g^n(P))$ is an irreducible variety. Since the map $f^n$ is also closed we have that $h$ is a closed and dominant mapping with an irreducible target. It follows that $h$ is surjective. So there is a point $Q\in \phi^{-1}(P)$ with $h(Q)=f^n(Q)=Q^\pp$. Since $\phi(Q)=P\in W$ we have that $Q\in \phi^{-1}(W)$ and so $f^n(\phi^{-1}(W))\cap U\neq \emptyset$. Thus $\OO_f(\phi^{-1}(W))$ is dense in $X$ as needed. On the other hand suppose that $\OO_f(\phi^{-1}(W))$ is dense in $X$. Let $V$ be open and non-empty in $Y$. Then $\phi^{-1}(V)\cap f^n(\phi^{-1}(W))\neq \emptyset$ as $\phi^{-1}(W)$ is dense in $X$. So there is a point $Q\in \phi^{-1}(W)$ with $f^n(Q)\in \phi^{-1}(V)$. Thus $g^n(\phi(Q))\in V$. Since $\phi(Q)\in W$ we have that $g^n(W)\cap V\neq \emptyset$ and $\OO_g(W)$ is dense in $Y$ as needed.    		\end{proof}
	
\begin{corollary}\label{cor:boundondiff}
Let $f\colon X\ra X$ be a surjective endomorphism of a normal $\QQ$-factorial variety with at worst terminal singularities. Let $\phi\colon X\ra Y$ be a contraction morphism and $g\colon Y\ra Y$ a surjective endomorphism with $g$ surjective and $g\circ \phi=\phi\circ f$. Suppose that $\phi$ is of fibering type. Then
\[\textnormal{Difficulty}(X,f)\leq \dim X-\dim Y+\textnormal{Difficulty}(Y,g)\]	
\end{corollary}
\begin{proof}
Let $W$ be a progenitor for $Y$ as in \ref{def:difficulty}.  As $\OO_g(W)$ is dense we have that the orbit of the generic point of $W$ is dense in $Y$. Thus we may assume that 
\[\dim \phi^{-1}(W)=\dim X-\dim Y+\dim W\]
as this is generically the case. By \ref{lem:MSfibrationbackbone} we have that $\OO_f(\phi^{-1}(W))$ is dense in $X$. Therefore one of the irreducible components of $\phi^{-1}(W)$ has a dense orbit in $X$. In other words 
\[\textnormal{Difficulty}(X,f)\leq \dim\phi^{-1}(W)=\dim X-\dim Y+\dim W. \]
Since $W$ is a progenitor for $Y$ with respect to $g$ we have that $\textnormal{Difficulty}(Y,g)=\dim W$ and consequently
\[\textnormal{Difficulty}(X,f)\leq \dim X-\dim Y+\textnormal{Difficulty}(Y,g). \]
\end{proof}
We obtain the following pleasing consequence of the definition of difficulty.
\begin{corollary}\label{cor:divandflipcor}
Let $f\colon X\ra X$ be a surjective endomorphism of a normal $\QQ$-factorial variety with at worst terminal singularities. Suppose that there is an $f$-equivariant MMP	\[\xymatrix{X=X_0\ar@{.>}[r]^{\phi_1}& X_1\ar@{.>}[r]^{\phi_2}&...\ar@{.>}[r]^{\phi_r}& X_r}\]
with respect to morphisms $f_i\colon X_i\ra X_i$. 
\begin{enumerate}
	\item Assume that each $\phi_i$ is divisorial or a flipping contraction and $X_r$ is a minimal model. Then 
	\[\textnormal{Difficulty}(X,f)=\textnormal{Difficulty}(X_r,f_r).\]
In other words, the difficulty can be computed on a minimal model.	
\item Suppose that $\phi_{r}\colon X_{r-1}\ra X_r$ is a Mori-fiber space and each $\phi_i$ for $i<r$ is a divisorial contraction or a flipping contraction. Then 
\[\textnormal{Difficulty}(X,f)\leq \dim X-\dim X_r+\textnormal{Difficulty}(X_r,f_r). \]
	\end{enumerate}
\end{corollary}		
\begin{proof}
 Assume that each $\phi_i$ is divisorial or a flipping contraction and $X_r$ is a minimal model. Then by \ref{lem:MCbirkey} and \ref{cor:MCbirkey} the difficulty is preserved by divisorial and flipping contractions. Thus $\textnormal{Difficulty}(X,f)=\textnormal{Difficulty}(X_r,f_r)$ as needed. Now suppose that $\phi_{r}\colon X_{r-1}\ra X_r$ is a Mori-fiber space and each $\phi_i$ for $i<r$ is a divisorial contraction or a flipping contraction. By the above argument we have that $\textnormal{Difficulty}(X,f)=\textnormal{Difficulty}(X_{r-1},f_{r-1})$. By \ref{cor:boundondiff} we have that
 \[\textnormal{Difficulty}(X_{r-1},f_{r-1})\leq \dim X_{r-1}-\dim X_r+\textnormal{Difficulty}(X_{r},f_{r}).\]
 Since the $\phi_i$ are birational for $i<r$ we have that $\dim X=\dim X_{r-1}$ and so
 \begin{align*}
 &\textnormal{Difficulty}(X,f)=\textnormal{Difficulty}(X_{r-1},f_{r-1})\\&\leq \dim X_{r-1}-\dim X_r+\textnormal{Difficulty}(X_{r},f_{r})\\&= \dim X-\dim X_r+\textnormal{Difficulty}(X_{r},f_{r}) 	
\end{align*} 
as claimed. 
\end{proof}		
The above corollary shows that the notion of difficulty is reasonably well behaved in the presence of an equivariant MMP. We now return to the Medvedev-Scanlon conjecture.  		
		\begin{lemma}\label{lem:rationalfiblem}
			Let $f\colon X\ra X$ be a surjective endomorphism of a normal $\QQ$-factorial variety with at worst terminal singularities. Suppose that $f$ does not preserve a rational fibration. If $f$ admits an equivariant MMP
			\[\xymatrix{X=X_0\ar@{.>}[r]^{\phi_1}& X_1\ar@{.>}[r]^{\phi_2}&...\ar@{.>}[r]^{\phi_r}& X_r}\]
			then $f_i\colon X_i\ra X_i$ does not preserve a rational fibration
		\end{lemma}
		\begin{proof}
			This follows from the fact that all the morphisms commutes. In particular if $f_i$ preserves a rational fibration then we have a diagram
			\[\xymatrix{X\ar[rr]^f\ar@{.>}[d] && X\ar@{.>}[d]\\ X_i\ar[rr]_{f_i}\ar@{.>}[dr]&&X_i\ar@{.>}[dl]\\ & Z &}\]
			This contradicts the fact that $f$ does not preserve a rational fibration. 
		\end{proof}
		\begin{theorem}
			Let $f\colon X\ra X$ be a surjective endomorphism of a normal $\QQ$-factorial variety with at worst terminal singularities. Suppose that $f$ admits an equivariant MMP
			\[\xymatrix{X=X_0\ar@{.>}[r]^{\phi_1}& X_1\ar@{.>}[r]^{\phi_2}&...\ar@{.>}[r]^{\phi_r}& X_r}\]
			where $\phi_i$ is a divisorial or flipping contraction for $i<r-1$ and $\phi_{r}\colon X_{r-1}\ra X_r$ is a Mori-fiber space. Suppose that $\phi_{r-1}$ has relative difficulty $0$. If the Medvedev-Scanlon conjecture holds for $X_r$ then it holds for $X$.
			\end{theorem}
		\begin{proof}
		If $f$ preserves a rational fibration then there is nothing to show. Suppose that $f$ does not preserve a rational fibration. Then $f_r\colon X_r\ra X_r$ does not preserve a rational fibration by \ref{lem:rationalfiblem} and the assumption that $X_r$ satisfies the Medvedev-Scanlon conjecture we have that $f_r$ has a point with dense forward orbit. By \ref{lem:MSfibrationbackbone} we have that $f_{r-1}$ has a fiber $\phi_{r}^{-1}(y)$ with a dense orbit. By assumption $\phi_{r}$ has relative difficulty $0$ and so 
			\[\textnormal{Difficulty}(\OO_{f_{r-1}}(\phi_{r}^{-1}(y)),f_{r-1})\leq 0.\]
			Since the orbit is dense we have that $\textnormal{Difficulty}(X_{r-1},f_{r-1})=0$ and there is a dense orbit of a point under $f_{r-1}$.  Applying \ref{lem:MCbirkey} and \ref{cor:MCbirkey} we obtain that $f$ has a dense forward orbit as needed.
			\end{proof}	
		This leads to the following question:
		\begin{question}
			\normalfont Consider a diagram
			\begin{equation}
				\xymatrix{X\ar[r]^f\ar[d]_\phi& X\ar[d]^\phi\\ Y\ar[r]_g & Y}
			\end{equation}
			where $f$ is surjective and $\phi\colon X\ra Y$ a Mori-fiber space. Under what conditions do we have that the relative difficulty of $f$ with respect to $\phi$ is at most zero. In other words when is it the case that orbits of fibers of $\phi$ under $f$ can always be generated by a point. To study this situation fix a point $p\in Y$ and consider $Z_p=\overline{\OO_g(p)}$ and $V_p=\overline{\OO_f(\phi^{-1}(p))}$. Then we have a diagram
			\[\xymatrix{V_p\ar[r]^f\ar[d]_\phi& V_p\ar[d]^\phi\\ Z_p\ar[r]_g & Z_p}\]
			In this situation $Z_p$ now has a canonical dense orbit and $V_p$ is built out of the fibers of a fibering contraction and we look for a dense orbit of a point under $f$. Since $g$ cannot preserve a rational fibration, neither can $f\colon V_p\ra V_p$ so the Medvedev-Scanlon conjecture predicts that we can always find a point with dense orbit for $f$ in this situation. 
			\end{question}
	We now turn to the case that $X$ admits an int-amplified endomorphism. We give a reduction to two special cases which represent the current bottleneck for the Medvedev-Scanlon conjecture for varieties admitting an int-amplified endomorphism. 
\begin{theorem}\label{thm:intamproadblock}
			 Assume the following.
			 \begin{enumerate}
				\item The Medvedev-Scanlon conjecture holds for surjective endomorphisms of all $Q$-abelian varieties.
				\item  Consider a diagram
				\begin{equation}
					\xymatrix{X\ar[r]^f\ar[d]_\phi& X\ar[d]^\phi\\ Y\ar[r]_g & Y}
				\end{equation}
				where we assume the following:
				\begin{enumerate}
					\item $X$ is a normal projective $\QQ$-factorial variety with at worst terminal singularities.
					\item $X$ admits an int-amplified endomorphism.
					\item $\phi\colon X\ra Y$ is a Mori-fiber space and $f$ is surjective.
					\end{enumerate}
			Then $f$ has relative difficulty $0$ with respect to $\phi$.
			\end{enumerate} Then the Medvedev-Scanlon conjecture holds for surjective endomorphisms of  $\QQ$-factorial normal projective varieties with at worst terminal singularities that admit an int-amplified endomorphism.
		\end{theorem}
		\begin{proof}
			Let $f\colon X\ra X$ be a surjective endomorphism. Suppose that $f$ does not preserve a rational fibration. Since $X$ admits an int-amplified endomorphism by \ref{thm:MengZhangMMP} we have an $f^n$ equivariant 
		\[X=X_1\dashrightarrow X_2\dashrightarrow X_3\dashrightarrow\dots\dashrightarrow X_r\] where $X_r$ is a $Q$-abelian variety. Since $f$ does not preserve a rational fibration, $\kappa(X)\leq 0$ as otherwise $f$ preserves the Iitaka fibration. Suppose first that $\kappa(X)=0$. Since $\kappa(X)=0$ we have that no fibering contractions occur and that $\dim X_r>0$. Since $f$ preserves no rational fibration neither does $f_{r}$ and so as we have assumed the Medvedev-Scanlon conjecture for $Q$-abelian varieties we have that $f_{r}$ has a point with dense orbit. As all the other contractions are birational we may apply \ref{lem:MCbirkey} and \ref{cor:MCbirkey} repeatedly to obtain that $f^n$ and so $f$ has a dense orbit as required. Now suppose that $\kappa(X)<0$. We now induct on the Picard number of $X$. If $X$ has Picard number 1 then the equivariant MMP exhibits $X$ as a Mori-fiber space over a point after a finite sequence of flips. By \ref{cor:MCbirkey} we may assume that $f$ is a Mori-fiber space over a point. In other words we have a diagram \[\xymatrix{X\ar[r]^{f^n}\ar[d]_\phi& X\ar[d]^\phi\\ Y\ar[r]_g & Y}\] where $Y$ is a point and $g$ is the identity. By assumption (2) we have that $f$ has relative difficulty $0$ with respect to $\phi$. As any point has difficulty $0$ we have that $f$ has a dense orbit as required. Now suppose that the Picard number of $X$ is larger then one. We have a diagram
			\[\xymatrix{X_1\ar[r]^{f^n}\ar@{.>}[d]_{\phi_1}& X_2\ar@{.>}[d]^{\phi_1}\\ X_2\ar[r]_{f_2} & X_2}.\]
			If $\phi_1$ is not a fibering type contraction then by induction we have that the Medvedev-Scanlon conjecture holds for $f_2$. Since $\phi$ is birational we have by \ref{lem:MCbirkey} and \ref{cor:MCbirkey} that $f^n$ has a point with dense orbit and therefore so does $f$. We may now assume that $\phi_1$ is a Mori-fiber space. Now consider the diagram
			\[\xymatrix{X_1\ar[r]^{f^n}\ar[d]_{\phi_1}& X_2\ar[d]^{\phi_1}\\ X_2\ar[r]_{f_2} & X_2}.\]
			Then $X_2$ is normal and $\QQ$-factorial with at worst terminal singularities. Furthermore, $X_2$ admits an int-amplified endomorphism. To see this note that by assumption $X$ admits an int-amplified endomorphism, say $h$. Then there is some $m$ such that $\phi_1\circ h^m=h_2\circ \phi_1$ where $h_2\colon X_2\ra X_2$ is a surjective endomorphism. Then if $\lambda$ is an eigenvalue of $h_2^*\colon N^1(X_2)_\RR\ra N^1(X_2)_\RR$ we have that $\lambda$ is an eigenvalue of $h^m$. Recall that a surjective endomorphism is int-amplified if and only if every eigenvalue has absolute value strictly greater then one. So $\vert \lambda\vert=\vert \mu\vert ^m$ where $\mu$ is an eigenvalue of $h^*$. Since $\vert \mu\vert>1$ we have $\vert \mu\vert^n=\vert \lambda \vert>1$ and $h_2$ is int-amplified. Therefore $X_2$ is a normal $\QQ$-factorial variety with at worst terminal singularities that admits an int-amplified endomorphism. Furthermore, $f_2$ does not preserve a rational fibration as otherwise $f$ would as well by \ref{lem:rationalfiblem}. If $\kappa(X_2)= 0$ then by our earlier argument the Medvedev-Scanlon conjecture holds for $X_2$. On the other hand if $\kappa(X_2)<0$ then $X_2$ satisfies the inductive hypothesis and so the  Medvedev-Scanlon holds for $X_2$ in all cases. In particular, $f_2$ has a point with dense orbit. By assumption (2) we know that the difficulty of $f$ relative to $\phi_1$ is zero. Since $f_2$ has a point with dense orbit we have that $f^n$ has a fiber with a dense orbit. As the relative difficulty is zero this means that $f^n$ has a dense orbit by a point and so does $f$. Thus $f$ satisfies the Medvedev-Scanlon conjecture.  
		  
\end{proof}

Theorem \ref{thm:intamproadblock} shows that a possible attack on the Medvedev-Scanlon conjecture for varieties is to first prove the conjecture for $Q$-abelian varieties. Then verify that Mori-fiber spaces have relative difficulty zero. While this may seem daunting, showing that $Q$-abelian varieties satisfy the Medvedev-Scanlon conjecture would give partial results on the difficulty of $f$. The proof gives the following.

\begin{corollary}
Assume that the Medvedev-Scanlon conjecture holds for surjective endomorphisms of all $Q$-abelian varieties. Then the Medvedev-Scanlon conjecture holds for all $\QQ$-factorial normal projective varieties with at worst terminal singularities that satisfy the following two conditions:
\begin{enumerate}
	\item $X$ admits an int-amplified endomorphism.
	\item $\kappa(X)= 0$. 
\end{enumerate}
\end{corollary}

\end{section}
	
\begin{section}{Automorphisms of positive entropy and the Kawaguchi Silverman conjecture.}\label{sec:automorphisms}
		
In this final section we illustrate how the MMP can be used to obtain results in the Kawaguchi-Silverman conjecture. We will focus on automorphisms of varieties with finitely generated Nef cones. We first discuss some easy reductions to the Kawaguchi-Silverman conjecture for automorphisms that surprisingly seem to have not appeared in the literature, but are relatively easy. Recall that for any projective variety we have an exact sequence \[0\ra \Aut^0(X)\ra\Aut(X)\ra \pi_0\Aut(X)\ra 0. \] Here $\Aut^0(X)$ is the connected component of the identity element of $\Aut(X)$ and is a smooth connected algebraic group. See section one of \cite{autBrion} for an introduction to these notions. Thus $\pi_0\Aut(X)$ in part measures how far $\Aut(X)$ is from being a smooth connected algebraic group.
		\begin{lemma}[2.8 in \cite{autBrion}]
			Let $X$	be a projective variety. Then $\Aut^0(X)$ acts trivially on $N^1(X)_\RR$ by pull back.
			\end{lemma}
		\begin{proof}
			Fix a line bundle $L$. Then we have a morphism
			\[t_L\colon \Aut(X)\ra \Pic^0(X)\]
			given by 
			\[g\mapsto g^*L\otimes L^{-1}\]
			This defines a morphism of schemes $\Aut(X)\ra \Pic(X)$. Since $t_L(\textnormal{iden}_X))=\OO_X$ it takes the connected component of the identity of $\Aut(X)$ to the connected component of the identity of $\Pic^0(X)$. In other words we have that $g^*L\otimes L^{-1}\in \Pic^0(X)$ when $g\in \Aut^0(X)$. Since $L$ was arbitrary we have that
			\[g^* L\Num L\]
			for any line bundle $L$. Thus $\Aut^0(X)$ acts trivially on $N^1(X)$. as needed.
			\end{proof}
		
		\begin{corollary}
			Let $X$ be a projective variety and $g\colon X\ra X$ an automorphism. Then $\lambda_1(g)$ only depends on the equivalence class of $g$ in $\pi_0\Aut(X)$
		\end{corollary}
		\begin{proof}
			If $g$ and $g^\pp$ have the same class in $\pi_0\Aut(X)$ then there is some $h\in \Aut^0(X)$ with $gh=g^\pp$. Thus \[(g^\pp)^*=(gh)^*=h^*g^*=g^*\]
			since $h^*$ is the identity on $N^1(X)_\RR$. As $\lambda_1(f)$ is the spectral radius of the action of $f^*$ on $N^1(X)$ we have that $\lambda_1(g)=\lambda_1(g^\pp)$. 
		\end{proof}
		This raises the following question of realizability. 
		\begin{question}\label{ques:autorez}
			\normalfont Let $g\in \Aut(X)$ and $h\in \Aut^0(X)$.  Now set $g^\pp=gh$. Then the above argument shows that $g^*$ and $(g^\pp)^*$ have the same eigenvalues when acting on $N^1(X)$. Therefore the set of potential arithmetic degrees of $g$ and $g^\pp$ are the same. In general, is the set of arithmetic degrees the same? In other words if $\alpha_g(P)=\vert \mu\vert$ then is there some point $Q$ with $\alpha_{gh}(Q)=\vert\mu\vert$?
		\end{question}
In the setting of Question \ref{ques:autorez} we have that $\lambda_1(g)=\lambda_1(g^\pp)$. This means we will have points $P,Q$ such that \[\alpha_g(P)=\lambda_1(g)=\lambda_1(g^\pp)=\alpha_{g^\pp}(Q).\] Consequently we obtain that the \emph{maximum arithmetic degree} of an automorphism only depends on the class of the automorphism in the component group $\pi_0\Aut(X)$. The question can then be reduced to the following. Suppose that $g$ is an automorphism and $1<\alpha_g(P)<\lambda_1(g)$ for some point $P\in X(\Qb)$. Then for all $h\in \Aut^0(X)$ is there a point $Q\in X(\Qb)$ with $\alpha_{gh}(Q)=\alpha_g(P)$. This question should have a positive answer when $X$ is a smooth surface, for in that case the eigenvalues of automorphisms are of the form $\lambda_1(g),\lambda_1(g)^{-1},\mu_1,\dots \mu_s$ where $\vert \mu_i\vert=1$ by \cite[2.4.3]{MR3289919}. Thus the question for smooth projective surfaces is reduced to the case of the maximum arithmetic degree being an invariant of the class in the component group, which we know has a positive answer.

		We obtain the following easy result that says that the Kawaguchi-Silverman conjecture for automorphisms is only meaningful for varieties with a complicated automorphism group. Recall that it is common to say that an automorphism $f\colon X\ra X$ \textbf{has positive entropy} if $\lambda_1(f)>1$. 
		
		\begin{theorem}\label{thm:trivthm}
			Let $X$ be a normal projective variety defined over $\Qb$. If $\Aut(X)$ is an algebraic group then $X$ has no automorphism with positive entropy. In particular, the Kawaguchi-Silverman conjecture is trivially true for automorphisms of $X$.
		\end{theorem}
		\begin{proof}
			If $\Aut(X)$ is a algebraic group then $\Aut(X)$ has finitely many components. Since the components of $\Aut(X)$ are precisely the  cosets of $\Aut^0(X)$ we have that $\pi_0\Aut(X)$ must be finite. In other words given $f\in \Aut(X)$ we have that $f^N\in \Aut^0(X)$ for some $N$. Then we have that $f^N$ acts trivially on $N^1(X)$ as $\Aut^0(X)$ acts trivially on $N^1(X)$. Since the eigenvalues of $f^N$ are all one we have that for all eigenvalues $\lambda$ of $f^*$ acting on $N^1(X)$ we have that $\lambda^N=1$. In other words the eigenvalues of $f^*$ are all roots of unity and consequently we have that $\lambda_1(f)=1$. 
		\end{proof}
		
		We also have the following useful result.
		
		\begin{lemma}[2.10 in \cite{autBrion}]\label{lem:stablem}
			Let $X$ be a projective variety defined over $\Qb$. Let $L$ be an ample line bundle on $X$. Let $\Aut(X,[L]_\textnormal{Num})$ be the subgroup of all elements  $[f]\in \pi_0\Aut(X)$ with $f^*L\Num L$. Then  $\Aut(X,[L]_\textnormal{Num})$ is finite.	
		\end{lemma}
		One easily obtains the following.
		\begin{corollary}[2.11 and 2.12 in \cite{autBrion}]\label{cor:Brion}
			Let $X$ be a projective variety defined over $\Qb$.
			\begin{enumerate}
				\item The kernel of the action of $\pi_0 \Aut(X)$ on $N^1(X)$ is finite.
				\item  If the nef cone of $X$ is finitely generated and rational then $\Aut(X)$ is an algebraic group.
			\end{enumerate}
		\end{corollary}
		\begin{proof}
		We first prove (1). Fix an ample line bundle $L$. If $[f]\in \pi_0\Aut(X)$ and $[f]$ is in the kernel of the action of $\pi_0\Aut(X)$ on $N^1(X)$ then $[f]\in \Aut(X,[L]_\textnormal{Num})$ which by \ref{lem:stablem} is finite as needed. Now suppose that the nef cone of $X$ is finitely generated and rational It suffices to prove that $\pi_0\Aut(X)$ is finite. Fix $[f]\in \pi_0\Aut(X)$. Let $v_1,\dots v_r$ be the ray generators of $\Nef(X)$. We have that each $v_i$ is a primitive element of $N^1(X)$ in the sense that it is the first lattice point on the half line $\RR_{\geq 0}v_i$. Note that $f^*$ is an automorphism of the lattice $N^1(X)$. This is because $f^*$ is represented by an integral matrix, and so is its inverse $(f^{-1})^*$. Thus $\det f^*\circ (f^{-1})^*=1=\det f^*\det (f^{-1})^*$. It follows that $\det f\pm 1$. Therefore we must have that $f^*v_i=v_j$. Thus $f^*(\sum_{i=1}^rv_i)=\sum_{i=1}^rv_i$ and consequently $[f]$ preserves the ample class $\sum_{i=1}^rv_i$. By \ref{lem:stablem} we have that $\pi_0\Aut(X)$ is finite as needed.   
		\end{proof}
		We obtain the Kawaguchi-Silverman conjecture for automorphisms of any normal projective variety with finitely generated and rational nef cone.
	\begin{corollary}\label{cor:fgnefcone}
			Let $X$ be a normal projective variety over $\Qb$. If $X$ has a finitely generated and rational nef cone then $X$ has no automorphism of positive entropy. In particular, the Kawaguchi-Silverman conjecture trivially holds for all automorphisms of $X$.
		\end{corollary}
		\begin{proof}
			By \ref{cor:Brion} we have that $\Aut(X)$ is an algebraic group. By \ref{thm:trivthm} we have the result. 
		\end{proof}
		We see that the Kawaguchi-Silverman conjecture for automorphisms of varieties with a finitely generated and rational nef cone is trivial. However, this leads to the question about varieties with finitely generated but non-rational nef cone. For example in \cite{MR3263669} there are examples of a Hyper-Kahler with Picard number 2 and an infinite automorphism group. In this case an automorphism of positive entropy may arise. 
		
		We would like to now define a subgroup of $\Aut(X)$ as those automorphisms which have dynamical degree 1. However, we run into the following problem. It is possible to have invertible integer matrices $A,B$ with $\det A=\det B=1$ and $\rho(A)=\rho(B)=1$ but $\rho(AB)>1$ where $\rho$ is the spectral radius function. Therefore it is a priori possible that there are automorphisms $f,g\in \Aut(X)$ with $\lambda_1(f)=\lambda_1(g)=1$ but $\lambda_1(fg)>1$. There is a way to avoid this issue when $X$ has finitely generated nef cone.
		\begin{definition}\label{def:DX}
			Let $X$ be a projective variety over $\Qb$ with finitely generated nef cone. That is $\Nef(X)_\RR$ is generated as a cone by finitely many real classes. Suppose that $\Nef(X)_\RR$ has rays $v_1,\dots v_r$. Then any surjective endomorphism of $X$ permutes the rays of $\Nef(X)$. In particular we have a homomorphism \[\mathfrak{r}\colon \pi_0 \Aut(X)\ra S_r\] where $S_r$ is the symmetric group on $r$ letters. Let $d_1$ be the size of the image of $\mathfrak{r}$. So $d_1$ is the smallest integer such that for all $f$ we have that $(f^{d_1})^*v_i=\lambda_i v_i$ for all $i$ and some real numbers $\lambda_i$. On the other hand the kernel of the action of $\pi_0 \Aut(X)$ on $N^1(X)$ is finite by \ref{cor:Brion}. Let $d_2$ be the size of this kernel and let  $d=\textnormal{lcm}(d_1,d_2)$. Now define $\mathfrak{D}(X)$ to be the subgroup of $\pi_0\Aut(X)$ generated by all $2d^{th}$ powers. That is
			\[\mathfrak{D}(X)=\langle [f^{2d}]\colon [f]\in \pi_0\Aut(X)\rangle\subseteq \pi_0\Aut(X). \]
			\end{definition}
	We think of $\mathfrak{D}(X)$ as the subgroup of all classes of automorphisms $\pi_0\Aut(X)$ which are simultaneously diagonalizable with positive eigenvalues. The basic properties of this group are outlined below.
		\begin{proposition}\label{prop:DXprop}
			Let $X$ be a projective variety over $\Qb$ with finitely generated nef cone. 
			\begin{enumerate}
				\item If $f_1,f_2\in \mathfrak{D}(X)$ then $f_1^*,f_2^*$ are simultaneously diagonalizable.
				\item There is an homomorphism $\textnormal{Lin}\colon \mathfrak{D}(X)\ra \textnormal{diag}_{\rho(X)}(\RR_{>0})\cong (\RR_{>0}^*)^{\rho(X)}$ with finite kernel. Here $\rho(X)$ is the Picard number of $X$,  $\textnormal{diag}_{\rho(X)}(\RR_{>0})$ are diagonal $\rho(X)\times \rho(X)$ matrices with positive entries and $\textnormal{Lin}([f])=f^*\colon N^1(X)_\RR\ra N^1(X)_\RR$. 
				\item The kernel of $\textnormal{Lin}\colon \mathfrak{D}(X)\ra \textnormal{diag}_{\rho(X)}(\RR_{>0})$ is precisely the set of $f\in  \mathfrak{D}(X)$ with $\lambda_1(f)=1$. 
			\end{enumerate}	
		\end{proposition}
		\begin{proof}
			We write $f_1=\prod_{i=1}^{s_1} g_{i}^{2d}$ and $f_2=\prod_{i=1}^{s_2} h_i^{2d}$ where we use that $f_1,f_2$ represent classes in $\mathfrak{D}(X)$. Then $f_1^*=(g_{s_1}^{2d})^*\dots (g_{1}^{2d})^*$. Since each $(g_i^{2d})^*v_j=\lambda_{ij}v_j$ we have that  $f_1^* v_j=\prod_{i=1}^{s_1} \lambda_{ij}v_j$. Similarly we have  $(h_i^{2d})^*v_j=\mu_{ij}v_j$ so that $f_2^*v_j=\prod_{i=1}^{s_2}\mu_{ij}v_j$. Since some sub-set of the rays is a basis of $N^1(X)_\RR$ we have that $f_1^*$ and $f_2^*$ share a mutual basis of eigenvectors.
			
			Now suppose that $f\in \mathfrak{D}(X)$ and $\textnormal{Lin}(f)=\textnormal{identity}$. Then $f^*$ lies in the kernel of the action of $\pi_0\Aut(X)$ which is finite by \ref{cor:Brion}. On the other hand, the above calculation shows that for any $f_1,f_2$ we have that \[f_1^*f_2^*v_i=\left(\prod_{i=1}^{s_1}\lambda_{ij}\right)\cdot \left(\prod_{i=1}^{s_2}\mu_{ij}\right)v_i=f_2^*f_1^*v_i.\]
			Since some subset of the $v_i$ is a basis we have $f_1^*f_2^*=f_2^*f_1^*$ and so
			\[\textnormal{Lin}(f_1f_2)=(f_1f_2)^*=f_2^*f_1^*=f_1^*f_2^*=\textnormal{Lin}(f_1)\textnormal{Lin}(f_2)\]
			so $\textnormal{Lin}$ is a homomorphism as desired. Finally note that by the definition of $d$ we have that for any $f\in \pi_0\Aut(X)$ that $(f^d)^*v_i=\mu_i v_i$.  So $(f^{2d})^*v_i=\mu_i^2v_i$. Thus the eigenvalues of any $f\in \mathfrak{D}(X)$ are positive. Thus $\textnormal{Lin}(f)$ is a diagonal matrix with positive entries in the basis given by the rays with positive entries. Finally if $\textnormal{Lin}(f)$ is the identity then $\lambda_1(f)=1$. Conversely let $\lambda_1(f)=1$ with $f\in  \mathfrak{D}(X)$. Let $f^*v_i=\lambda_iv_i$ with $\lambda _i>0$. Then we have shown that $f^*$ is diagonal with eigenvalues $\lambda_i$. Then $
			\det f^*=\prod_i \lambda_i=1$. As $0<\lambda_1(f)\leq 1$ if some $\lambda_i<1$ then for the product to equal one we must have some $\lambda_j>1$. It follows that $\lambda_i=1$ for all $i$. Since $f^*$ is diagonal we have that $f^*$ is the identity. So $\ker \textnormal{Lin}=\{f\in \mathfrak{D}(X)\colon \lambda_1(f)=1 \}$ and consequently this set is finite. 
			\end{proof}
		We can now give a group theoretic criterion for when a variety has an automorphism with positive entropy in terms of the component group. Let $[f]\in \pi_0 \Aut(X)$. It is certainly a necessary condition for $\lambda_1(f)>1$ that $[f]$ have infinite order in $\pi_0\Aut(X)$. We show that this is in fact sufficient. In other words, the obvious necessary condition is also sufficient. 
		\begin{theorem}[Criterion for when a variety with finitely generated nef cone has an automorphism of positive entropy]\label{thm:DXthm}
			Let $X$ be a normal projective variety over $\Qb$ with a finitely generated nef cone. Let $f\in \pi_0\Aut(X)$. Then $\lambda_1(f)>1\iff$ f has infinite order in $\pi_0\Aut(X)$. In particular a normal projective variety $X$ with finitely generated nef cone has an automorphism of positive entropy if and only if $\pi_0 \Aut(X)$ has an element of infinite order.
		\end{theorem}	
		\begin{proof}
			If $f\colon X\ra X$ is an automorphism and $\lambda_1(f)>1$ then as $\lambda_1(f^n)=\lambda_1(f)^n$ we see that $f$ has infinite order. On the other hand suppose that $f$ has infinite order. Let $d$ be as in \ref{def:DX}. Then $f^{2d}\in \mathfrak{D}(X)$. Towards a contradiction suppose that $\lambda_1(f)=1$. Then $\lambda_1(f^{2d})=1$ and so $f^{2d}$ lies in the kernel of the action of $\pi_0\Aut(X)$ by \ref{prop:DXprop}. By \ref{cor:Brion} this group has finite order and $(f^{2d})^N=f^{2dN}=\textnormal{identity}$ contradicting that $f$ had infinite order. 
		\end{proof}
		Thus to produce examples of varieties with automorphisms of positive entropy it suffices to produce varieties with finitely generated nef cone with component group having an element of positive entropy. On the other hand, the component group is still currently a mysterious object. It was only recently shown by Lesieutre in \cite{MR3773792} that $\pi_0 \Aut(X)$ can be non-finitely generated. On the other hand, it is a folklore question that asks if there exists infinite finitely presented groups with every element of finite order. See \cite[Section 1]{0704.2899} or \cite{78410}.  One might ask if such a group can arise as the automorphism group of a projective variety with a finitely generated nef cone.  
		\begin{question}
			Let $X$ be a normal projective variety defined over $\Qb$ with finitely generated nef cone. Is it possible that $\pi_0\Aut(X)$ is finitely presented but $X$ has no automorphism of positive entropy? By \ref{thm:DXthm} this is equivalent to asking if $\pi_0 \Aut(X)$ can be finitely presented with no element of infinite order.
		\end{question}
	
We finally note that it may be useful to apply this same analysis to some of the other cones sitting inside $N^1(X)_\RR$ such as the closure of the big cone. Even if the nef cone is not finitely generated, perhaps one of these other cones could be and similar results could be applied.	
		
We now turn to the Kawaguchi-Silverman conjecture for automorphisms. 
		
		\begin{definition}
			Let $X$ be a $\QQ$-factorial normal projective variety with at worst terminal singularities and finitely generated not necessarily rational nef cone. \begin{enumerate}
				\item Let $\phi\colon X\ra Y$ be a small contraction and $\phi^+\colon X^+\ra Y$ an associated flip. We say that $\phi$ is  \textbf{polyhedral} if $X^+$ also has finitely generated nef cone.
				\item  Suppose that $X$ admits an MMP
				\[X=X_0\dashrightarrow X_1\dashrightarrow \dots\dashrightarrow X_r\]
				with each $X_{i}\dashrightarrow X_{i+1}$ either a divisorial, flipping, or fibering contraction associated to a $K_{X_i}$-negative extremal ray, and either $X_r$ is minimal or $X_{r-1}\ra X_r$ is a fibering contraction. We call the MMP \textbf{polyhedral} if each flipping contraction is polyhedral. 
				\item Let $\mathscr{P}$ be the category of all $\QQ$-factorial normal projective varieties with at worst terminal singularities and finitely generated not necessarily rational nef cone that \textbf{admit a polyhedral MMP}.  We let $\mathscr{P}_{-\infty}$ be the sub-category of $\mathscr{P}$ that admit a \textbf{tractable polyhedral MMP} ending at a point. That is all varieties $X$ which admit an MMP
				\[\xymatrix{X=X_0\ar@{.>}[r]^{\phi_1}& X_1\ar@{.>}[r]^{\phi_2}&...\ar@{.>}[r]^{\phi_r}& X_r}\]
				with each $\phi_i$ a divisorial, fibering, or polyhedral flipping contraction and $X_r$ a Q-abelian variety.
				\end{enumerate}
		\end{definition}
		We have the following easy result.
		\begin{theorem}\label{thm:autthm1}
			Suppose that the Kawaguchi-Silverman conjecture holds for automorphisms of minimal varieties with finitely generated nef cones. Then the Kawaguchi-Silverman conjecture for automorphisms holds for all varieties in $\mathscr{P}$.
		\end{theorem}
		\begin{proof}
			Let $X\in \mathscr{P}$. We induct on the Picard number $\rho(X)$. If $\rho(X)=1$ then the Kawaguchi-Silverman conjecture is true for all automorphisms of $X$. So we may assume $\rho(X)>1$. If $X$ is minimal we are done by assumption. So assume that $X$ is not minimal. Then as $X\in \mathscr{P}$ we have a polyhedral MMP 
			\[X=X_0\dashrightarrow X_1\dashrightarrow \dots\dashrightarrow X_r\]
			Since $X$ is not minimal and the contractions are contractions of $K_{X_i}$-negative extremal rays we have that $r>0$. Note that
			\[X_i\dashrightarrow X_{i+1}\dashrightarrow \dots\dashrightarrow X_r\]
			is a polyhedral MMP for all $i>0$. So each $X_i$ is in $\mathscr{P}$. This is because if $X\ra Y$ is a divisorial or fibering contraction then if $X$ has a finitely generated nef cone then so does $Y$ because the nef cone of $Y$ is the face of the nef cone of $X$. As $X$ has a finitely generated nef cone, all its faces are also finitely generated.  On the other hand since we have assumed that each flipping contraction is polyhedral, we are guaranteed that all flips $X^+$ that arise in the MMP also have finitely generated nef cone. Now let $f\colon X\ra X$ be an automorphism. As $\Nef(X)_\RR$ has dual cone $\NE(X)_\RR$ we have that the closed cone of curves is finitely generated. Then as $f_*$ preserves the closed cone of curves we have that $f_*$ permutes the rays of  $\NE(X)_\RR$. Thus for some $N>0$ we have that $f^N_*$ fixes the rays of $\NE(X)_\RR$. Let $\phi\colon X\ra X_1$ be the first contraction. Suppose that $\phi_1$ is induced by the contraction of an extremal ray $\mathfrak{R}$. Suppose first that $\mathfrak{R}$ gives a divisorial contraction.  Since $f^N_* \mathfrak{R}=\mathfrak{R}$ by \ref{lem:iterationlemma2} we have a diagram
			\[\xymatrix{X\ar[d]_{\phi_1}\ar[r]^{f^N}&X\ar[d]^{\phi_1}\\ X_1\ar[r]_{f_1}& X_1}\]  
			By induction the Kawaguchi-Silverman conjecture holds for $f_1$ and since $\phi_1$ is birational it also holds for $f^N$ and so $f$. Now suppose that $\mathfrak{R}$. By \cite[6.2]{1802.07388} we have a diagram
			\[\xymatrix{X\ar[d]_{\phi_1}\ar[r]^{f^N}&X\ar[d]^{\phi_1}\\ X_1\ar[r]_{f_1}& X_1}\]
		    with $f_1$ an automorphism. By induction the Kawaguchi-Silverman conjecture holds for $f_1$ and so by \cite[6.3]{1802.07388} the Kawaguchi-Silverman conjecture holds for $f^N$ and so for $f$ as well. Thus we may assume that $\mathfrak{R}$ is a flipping contraction. Since $f^N_* \mathfrak{R}=\mathfrak{R}$ we have that $f^N$ extends to a morphism $f^+\colon X^+\ra X^+$ that is birationally conjugate to $f$ by \ref{lem:extensionlem2}. So the Kawaguchi-Silverman conjecture for $f$ is equivalent to the Kawaguchi-Silverman conjecture for $f^+$. Now repeat the procedure for $f^+$. Either we eventually arrive at a divisorial or fibering contraction and apply the earlier arguments, or the MMP is a series of polyhedral flips terminating at a minimal model. By assumption the Kawaguchi-Silverman conjecture holds for minimal models and so for $f^N$ and consequently $f$.   
			\end{proof}
		The ideas in the proof give an argument for the triviality of the Kawaguchi-Silverman conjecture for varieties in  $\mathscr{P}_{-\infty}$. 
		\begin{corollary}\label{cor:autcor1}
			Let $X\in \mathscr{P}_{-\infty}$. Then $X$ has no automorphism with positive entropy.  In particular the Kawaguchi-Silverman conjecture holds and every element of $\pi_0\Aut(X)$ has finite order. 
		\end{corollary}
		\begin{proof}
			Let $X\in \mathscr{P}_{-\infty}$. We induct on the Picard number $\rho(X)$. If $\rho(X)=1$ then the Nef cone of $X$ is finitely generated and rational. So by \ref{cor:fgnefcone} we have that $X$ has no automorphism of positive entropy. Now let $\rho(X)>1$.  A variety in  $\mathscr{P}_{-\infty}$ has a tractable polyhedral MMP \[X=X_0\dashrightarrow X_1\dashrightarrow \dots\dashrightarrow X_r\] ending in a point. Arguing as in the proof of \ref{thm:autthm1} we eventually have a diagram
			\[\xymatrix{X_i\ar[r]^{f_i}\ar[d]_{\phi_i}& X_i\ar[d]^{\phi_i}\\ X_{i+1}\ar[r]_{f_{i+1}}& X_{i+1}}\]
			where $f_i$ and $f_{i+1}$ are automorphisms, and $\phi_i$ is a fibering or divisorial contraction. Moreover $f_i^*$ and $f^*$ have the same eigenvalues and $\rho(X)=\rho(X_i)$. Note that $\rho(X_{i+1})=\rho(X_i)-1=\rho(X)-1$. By induction we see that $f_{i+1}$ does not have positive entropy. Let $\mu_1,...,\mu_{\rho(X)-1}$ be the eigenvalues of $f_{i+1}^*$. As $\lambda_1(f_{i+1})=1$ we have $\vert \mu_k\vert\leq 1$ for all $k$. Since the diagram above commutes and  $\rho(X_i)-1=\rho(X_{i+1})$  we have that the eigenvalues of $f_i^*$ are the eigenvalues of $f_{i+1}$ along with a single potentially new eigenvalue $\gamma$. It suffices to show that $\vert\gamma\vert\leq 1$. We have that
			\[1=\vert\det f_{i}^*\vert=\vert \gamma \cdot \mu_1\dots \mu_k \vert=\vert \lambda\vert\]
			as needed. So $X$ has no automorphism of positive entropy. By \ref{thm:DXthm} we have that $\pi_0 \Aut(X)$ does not contain an element of infinite order. 
			\end{proof}
		\end{section}	

	\bibliographystyle{plain}
\bibliography{bib}
\end{document}